\documentclass{amsart}

\usepackage[left=2.5cm,right=2.5cm,top=4cm,bottom=4cm]{geometry}
\usepackage{amstext, amssymb, amsthm, amsfonts, latexsym,bbm,tikz}
\usepackage{mathtools}
\usetikzlibrary{matrix}
\usepackage{varioref}
\usepackage{float}

\usepackage{todonotes}
\usepackage{bbm}
\usepackage[utf8]{inputenc}
\usepackage[english]{babel}
 
\usepackage{comment}
\usepackage{tikz-cd}
\usepackage{hyperref}
\usepackage[version=3]{mhchem}

\newtheorem{theorem}{Theorem}[section]
\newtheorem{lemma}[theorem]{Lemma}
\newtheorem{corollary}[theorem]{Corollary}
\newtheorem{proposition}[theorem]{Proposition}

\theoremstyle{definition}
\newtheorem{definition}[theorem]{Definition}
\newtheorem{example}[theorem]{Example}

\theoremstyle{remark}
\newtheorem{remark}[theorem]{Remark}

\theoremstyle{Conjecture/open problem}

\theoremstyle{assumption}
\newtheorem{assumption}{Assumption}

\theoremstyle{conjecture}
\newtheorem{conjecture}{Statement}

\newcommand{\intval}{q}

\newcommand{\st}{\,\mid \,}
\def\la{\lambda}
\def\I{\mathcal{I}}
\def\Pr{\mathbb{P}}
\def\a{\alpha}
\def\be{\beta}

\def\Ga{\Gamma}
\def\1{\mathbf{1}}
\def\la{\lambda}

\def\p{^{\prime}}
\def\pp{^{\prime\prime}}

\def\k{\kappa}

\def\cG{\mathcal{G}}
\def\cR{\mathcal{R}}
\def\cS{\mathcal{S}}
\def\cC{\mathcal{C}}

\def\R{\mathbb{R}}
\def\N{\mathbb{N}}
\def\Z{\mathbb{Z}}

\numberwithin{equation}{section}

\begin{document}

\title[An algebraic approach to stochastic CRN]{An algebraic approach to product-form stationary distributions for some reaction networks}

\author[Beatriz Pascual-Escudero]{Beatriz Pascual-Escudero}
\address{Department of Mathematical Sciences, University of Copenhagen, Denmark}
\email{beatriz.pascual.escudero@gmail.com}

\author[Linard Hoessly]{Linard Hoessly}
\address{Department of Mathematical Sciences, University of Copenhagen, Denmark}
\email{hoessly@math.ku.dk}


\begin{abstract}
Exact results for product-form stationary distributions of Markov chains are of interest in different fields. In stochastic reaction networks (CRNs), stationary distributions are mostly known in special cases where they are of product-form. However, there is no full characterization of the classes of networks whose stationary distributions have product-form.
We develop an algebraic approach to product-form stationary distributions in the framework of CRNs. Under certain hypotheses on linearity and decomposition of the state space  for conservative CRNs, this gives sufficient and necessary algebraic conditions for product-form stationary distributions. Correspondingly we obtain a semialgebraic subset of the parameter space that captures rates where, under the corresponding hypotheses, CRNs have product-form. We employ the developed theory to CRNs and some models of statistical mechanics, besides sketching the pertinence in other models from applied probability.
\end{abstract}

\maketitle

\section{Introduction}\label{intro_s}
 Many stochastic processes of interest count abundances of different entities, and serve as basic models in applied or theoretical probability. Often, such models can be expressed as
 continuous-time Markov chains (CTMCs) on $\Z_{\geq 0}^n$. In terms of examples, consider, e.g., reaction networks \cite{mcquarrie_1967}, interacting particle systems \cite{liggett}, ecology \cite{ecol_May}, networks in the sciences \cite{Goutsias_ov} or queuing networks \cite{serfozo}. In this work, we study the stationary distributions of closed systems of this type, 
 focussing on conditions for product-form. In particular, our results exhibit that in different cases of interest the subset of the parameter space\footnote{I.e. for the parameters involved in the stochastic model in question.} with product-form stationary distribution is semi-algebraic. While we focus on stochastic CRNs, the developed theory applies broadly to similar models, i.e. as long as the parametrisation of the CTMC model  is linear and has a certain state space decomposition (see Sections \ref{subs_other_m}). 

\subsection{Reaction networks}

CRNs are an effective way  
to describe interactions of different species through mathematical models in the sciences.
They are of interest in 
 biochemistry, systems biology and cellular biology, 
besides 
other applications \cite{Goutsias_ov,overv_Gorban,gardiner}.

A CRN consists of a set of reactions, that come together with associated reaction rates that govern their speed. They are usually expressed via their reaction graph. 
As an example, consider the reversible Michaelis-Menten mechanism:
\begin{equation}
S+E\rightleftharpoons ES\rightleftharpoons P+E.\end{equation}
Two approaches are adopted when it comes to studying the dynamics of CRNs. 
Deterministic models mostly consist 
 of ordinary differential equations (ODE) which describe  
 the changes in species' concentrations in time.
If the molecular counts in the system are low, stochastic models like CTMCs are usually better suited to describe the dynamics. Such stochastic and deterministic systems can be analysed in terms of the transient behavior, i.e. their evolution in time. On the other hand, the stationary behavior is of interest as well, e.g. the description of the system when it has reached an equilibrium, if it does. 
In the following we mostly focus on stochastic CRNs and their stationary behavior.

\subsection{Stochastic reaction networks and analytical results on stationary distributions}
Traditionally, deterministic models
have been the preferred modelling choice. However, 
the analysis of living systems with small molecular counts and stochasticity has become essential. Correspondingly there is an increased interest in the dynamics of the CTMC for CRNs. 
However, due to the difficulty in analytical tractability, most studies are based on approximations or simulations \cite{gardiner,Saito1,Saito2,gillespieb}. Distinctive differences separate the stochastic and deterministic models, both in terms of their transient and stationary behavior. Besides the obvious differences, studied phenomena include, e.g., the small number effect \cite{Saito1}, absorption in absolutely robust CRNs \cite{Anderson14}, condensation \cite{hoesslysta}, or discreteness induced transitions \cite{Saito2}. 

Characterising transient and stationary behaviour of stochastic CRNs are \\
formidable tasks in general, and they are often examined via simulations \cite{gillespieb}. 
Analytical solutions for the stationary distribution (in case of existence) are unknown for most systems, except for some special cases. Some stationary distributions of CRNs are well-understood. Both complex balanced  \cite{anderson2} and autocatalytic CRNs  \cite{hoesslysta} have product-form stationary distributions (see e.g. Definition \ref{prod_f}).
Furthermore \cite{hoesslyuni} offers sufficient conditions for product-form stationary distributions through decompositions of CRNs as well as different examples. Note that solutions  of product-form stationary distributions can be used to prove positive recurrence, see, e.g. \cite[II, Corollary 4.4]{asmussen2003applied} for the general case, or \cite{non-expl} for the application to complex balance. 
Beyond these results, little is known concerning explicit stationary distributions. Analytical results on product-form stationary distributions allow convenient expressions for the long-term dynamics \cite{anderson2,hoesslyuni,liggett,Schadschn_cx,serfozo}. Furthermore they often enable the rigorous study of scaling limits, condensation, or other features of the dynamics \cite{hoesslysta,Cao2,kipnis1999sli}.

Previous results were mostly focussed on particular classes of stochastic CRNs for which the stationary distributions can be given, i.e. \cite{anderson2,non-stand_2,hoesslysta,hoesslyuni}. Here, we aim to give tools to systematically find regions of the reaction rate space 
where the corresponding stochastic CRN has or has not product-form stationary distribution. The underlying idea is that the condition for product-form stationary distributions for closed CTMC dynamics on $\Z_{\geq 0}^n$ with linear parametrisations is canonically related to conditions expressible through algebraic geometry.  This enables to study the semi-algebraic set of product-form for corresponding CRNs computationally via a bottom-up approach.
\subsection{Overview and main results}
 We recall in Section \ref{RN} main concepts of stochastic CRNs, with corresponding notions as primary examples and motivation for our study. The theoretical part in Section \ref{geom_v} then consists of the following.
\begin{enumerate}
\item We show that stationary distributions have polynomial parametrisations, as long as the reaction rates appear linearly in the kinetics (Lemma \ref{clm}).
\item We give necessary, and sometimes equivalent, conditions for product-form stationary distributions for sequences of probability distributions along irreducible components of specific forms (Theorem \ref{suff_condn}).
\item In Section \ref{semi} we combine the previous findings to define an ideal and a semi-algebraic set that characterises CRNs with product-form stationary distribution, under certain assumptions on the parametrisation of the reaction rates and the structure of the irreducible components.
\end{enumerate}
While for some conservative CRNs, the conditions we obtain are sufficient and necessary, for others they are only necessary, or do not apply (e.g. the infinite case). On the other hand the results derived do not depend on the kinetics at hand, and the methods can be applied to other models of CTMCs. 
We analyse some classes of examples from CRN theory and particle systems in Section \ref{exa} and sketch some further applications, ending with a discussion of our considerations in Section \ref{disc}.

\subsection{Notation}
We write $[k]$ for the set $\{1,\cdots , k\}$, $\pi$ for the stationary distributions and $Z$ the normalising constants, when we do not further specify the domain. $\Ga$ denotes an irreducible component, and if we want to specify the domain we write $\pi_\Ga$ for the stationary distribution on $\Ga$ and $Z_\Ga$ for the corresponding normalising constant.

For a finite set $\Ga$, 
let $\Delta_{\Ga}$ stand for the probability simplex, i.e.
$$\Delta_{\Ga}:=\{x\in \R^\Ga| \sum_{i\in\Ga}x_i=1,x_i\geq 0 \text{ for }i\in\Ga\}, $$
and let $\mathring \Delta_{\Ga}$ be the ordinary interior of $\Delta_{\Ga}$, consisting of probability distributions which are nonzero on all coordinates.
Let us denote by $\R\Pr^{\Ga}$, $\R\Pr^{\Ga}_{\geq 0}$ and $\R\Pr^{\Ga}_{> 0}$ the 
real, nonnegative real and positive real projective space on $\Ga$ respectively.
We write $e_i$ for the vector with $i$-th entry equal to one and all other entries equal to zero.

\section{Stochastic reaction networks}\label{RN} \label{modelCRN}
\subsection{Basic terminology for CRNs}\label{reaction_basic}

A \emph{reaction network} consists of three finite sets $\cG=(\cS,\cC,\cR)$ where
 $\cS$ is the set of \emph{species} $\cS=\{S_1,\cdots,S_n\}$, $\cC$ is the set of \emph{complexes} and $\cR\subseteq \cC\times\cC$ is the set of \emph{reactions}.  

Complexes correspond to nonnegative linear combinations of species over $\Z_{\geq 0}$, which we write in the form
$\nu=\sum_{i=1}^n\nu_i S_i$, and identify with vectors $\nu\in\Z_{\geq 0}^n$. 
 Reactions consist of ordered tuples $(\nu, \nu\p)\in \cR$. The sets $\cC$ and $\cR$ are such that $(\nu , \nu )\not\in \cR$ and that if $\nu \in \cC$, there exists $\nu \p \in\cC$ such that either $(\nu , \nu \p)\in\cR$, or $(\nu \p,\nu )\in\cR$. The \emph{stochiometric subspace} is defined as the linear subspace $\mathcal{T}={\rm span}_{(\nu , \nu \p)\in\cR}\{\nu\p-\nu\}\subset \mathbb{R}^n$.  A CRN
  $\cG=(\cS,\cC,\cR)$  is \emph{conservative} (resp. subconservative) if there is a vector of positive weights $w\in \R^\cS_{>0}$ such that for any reaction $(\nu , \nu \p)\in\cR$ we have that $\sum_{S_i\in\cS}\nu_i w_i=\sum_{S_i\in\cS}\nu\p_i w_i$ (resp. $\sum_{S_i\in\cS}\nu_i w_i\geq\sum_{S_i\in\cS}\nu\p_i w_i$).

We usually describe a reaction network by its \emph{reaction graph}, which is the directed graph with vertices $\cC$ and edge set $\cR$. This way, we represent a reaction $(\nu , \nu \p)\in\cR$ as $\nu \to \nu \p$ and say that it consumes the \emph{reactant} $\nu$ and creates the \emph{product} $\nu\p$. We also say $\nu \in \cC$ \emph{reacts} to $\nu\p \in \cC$ if $\nu\to \nu\p$ is a reaction. The \emph{molecularity} of a reaction $\nu\to \nu\p \in\cR$ is equal to the number of molecules in the reactant $|\nu|=\sum_i \nu_i$. A connected component of the reaction graph of $\cG$ is termed a \emph{linkage class}. A CRN is \emph{reversible} if $\nu\to \nu\p \in\cR$ whenever $ \nu\p\to \nu \in\cR$. It is \emph{weakly reversible} if for any reaction $\nu\to \nu\p \in\cR$, there is a directed path in the graph beginning with $\nu\p$ as a reactant and ending with $\nu$ as a product complex. If it is not weakly reversible we say it is \emph{non-weakly reversible}. The \emph{deficiency}  of a reaction network $\cG$ is defined as $\delta=|\cC|-\ell-{\rm dim}(\mathcal{T}),$
 where $\ell$ is the number of linkage classes.  
	
	 Each reaction $\nu\to \nu\p$ has a positive \emph{reaction rate constant} (or simply \emph {rate}) $\k_{\nu\to\nu\p}$; the  vector of reaction rates is defined by $\k\in\R_{>0}^\cR$
 and the CRN with rates $\k $ is denoted by $(\cG,\k)$.

We will often refer to two specific classes of CRNs: Complex balanced CRNs which are defined e.g. in \cite[Section 3.2]{anderson2}, and autocatalytic CRNs which are defined in \cite[Section 3]{hoesslysta}. We note that the former are always weakly reversible, while the later are mostly non-weakly reversible.

  \subsection{Stochastic model of reaction networks}\label{stoch_model}
 The evolution of species counts follows a Markov process, where the vector $X(t)=x\in \Z_{\geq 0}^n$ changes, according to transitions determined by the reactions $\nu\to \nu\p$,  by jumping from $x$ to $x+\nu\p-\nu$  with transition intensity $\la_{\nu \to \nu\p}(x)$.
The transition intensity functions $\la_{\nu \to \nu\p}:\Z_{\geq 0}^n\to \R_{\geq 0}$ then give the $Q$-matrix of the form
$$Q(x,x+\xi):=\sum_{\nu\to\nu\p\in \cR\st-\nu + \nu\p=\xi}\la_{\nu \to \nu\p}(x)$$
such that the Markov process satisfies
$$P(X(t+\Delta t)=x+\xi|X(t)=x)=\sum_{\nu\to\nu\p\in \cR\st-\nu + \nu\p=\xi}\la_{\nu \to \nu\p}(x)\Delta t+ o(\Delta t).$$

The transition intensity function under Mass-action kinetics (more general kinetics are possible as well \cite{anderson2,non-stand_2}) associated to the reaction $\nu\to\nu\p$  is 
\begin{equation}
\label{int}\la_{\nu \to \nu\p}(x)=\k_{\nu\to\nu\p}\frac{(x)!}{(x-\nu)!}
\1_{x\geq \nu},
\end{equation}
$\text{ where }z!:=\prod_{i=1}^nz_i!\text{ for } z\in \Z^n_{\geq 0}$ and $\1_{x\geq \nu}$ is one if and only if $x_i\geq \nu_i$ for all $ i\in [n]$ and zero otherwise.

Stochastic CRNs are analysed via inspection of the underlying CTMC, where the state space is decomposed into different types of states (i.e. transient, recurrent, positive recurrent, see, e.g., \cite[Sections 3.4, 3.5]{norris}).

 We next introduce some terminology for stochastic CRNs:

A reaction $\nu \to \nu \p$ is \emph{active} on $x\in \Z_{\geq 0}^n$ if $x_i\geq \nu_i $ for all $ i\in [n]$. A state $u\in \Z_{\geq 0}^n$ is \emph{accessible} from another state $x\in \Z_{\geq 0}^n$ if $u$ is reachable from $x$ via its CTMC dynamics. We denote this by $x\to_{\cG}u$.

 A non-empty set $\Ga\subset\Z_{\geq 0}^n$ is an \emph{irreducible component} of $\cG$ if for all $x\in \Ga$ and all $u\in  \Z_{\geq 0}^n$, $u$ is accessible from x if and only if $u\in \Ga$.

  An irreducible component $\Ga$ is \emph{positive} if for all reactions there is some $x\in \Ga$ such that the reaction is active on that state.
An irreducible component is \emph{positive recurrent} if it contains a state that is positive recurrent, as then all states in the irreducible component are positive recurrent (see e.g. \cite[Theorem 3.5.3]{norris}).
	
 We say $\cG$ is \emph{essential} if the state space is a union of irreducible components, and

 $\cG$ is \emph{almost essential} if the state space is a union of irreducible components except for a finite number of states.

On irreducible components, positive recurrence (also called ergodicity) is equivalent to non-explositivity together with existence of an invariant distribution \cite[Theorem 3.5.3]{norris}, where non-explositivity defined in \cite[Section 2.7]{norris}. Non-explositivity corresponds to the CTMC not having infinitely many jumps over a compact time interval. The classification of stochastic CRNs is challenging in general. While different results exist like e.g. for positive recurrence \cite{andersonpos},  non-explositivity of complex balanced CRN \cite{non-expl}, extinction/absorption events \cite{extinction,quasi_stat}, quasi-stationary distributions \cite{quasi_stat} or 1-d stochastic CRNs \cite{xu2019dynamics}, we are far from a complete characterization for most.

\subsection{Stationary distributions of reaction networks} \label{stat}
Let $X(t)$ denote the underlying stochastic process from a CRN on an irreducible component $\Ga$. Then, given that the stochastic process $X(t)$ is positive recurrent on $\Ga$ and starts in $\Ga$, we have that the limiting distribution is the stationary distribution, i.e.,
$$\lim_{t\to\infty}P(X(t)\in A)=\pi_\Ga(A),\text{for any } A\subset\Ga.$$
Then, the stationary distribution $\pi_\Ga$ on an irreducible component $\Ga$ is unique and describes the long-term behavior \cite[Theorem 3.6.3]{norris}.  The stationary distribution is (implicity) determined by the Master equation of the Markov chain:
\begin{equation} \label{master_eq1}
  \sum_{\nu\to\nu\p\in\cR} \pi(x+\nu-\nu\p)\la_{\nu \to \nu\p}(x+\nu -\nu\p)=\pi(x)\sum_{\nu\to\nu\p\in\cR}\la_{\nu \to \nu\p}(x),
  \end{equation}
for all $x\in\Ga$. We denote by $ME(\cG,\Ga)$ the set of equations derived from the Master equation on the irreducible component $\Ga$ in terms of the reaction rates.

Under Mass-action kinetics, the Master equation then takes the following form:

\begin{equation} \label{master_eq2}  \sum_{\nu\to\nu\p\in\cR} \pi(x+\nu-\nu\p)\k_{\nu\to\nu\p}\frac{(x-\nu\p+\nu)!}{(x-\nu\p)!}\1_{x\geq \nu\p}=\pi(x)\sum_{\nu\to\nu\p\in\cR}\k_{\nu\to\nu\p}\frac{(x)!}{(x-\nu)!}\1_{x\geq \nu}.\end{equation}

Solving equation (\ref{master_eq1}) is in general a challenging task, even when restricting to the Mass-action case. 
\begin{remark}\label{rmk_pos_cons}
Sub-conservative CRNs have finite irreducible components, such that their CTMC dynamics are non-explosive \cite[Theorem 2.7.1]{norris}. Hence on irreducible components they are positive recurrent with unique solution to the Master equation \cite[Theorem 3.5.3]{norris}. 
Note that for infinite CTMCs, existence of stationary distribution does not imply positive recurrence, cf., e.g. \cite[Ex 3.5.4]{norris}. 
\end{remark}

Consider a reaction network $(\cG,\k)$ with corresponding CTMC dynamics that has at least one nontrivial\footnote{I.e. with more than one element.} and positive recurrent irreducible component. \begin{definition}\label{prod_f}
We say:\begin{enumerate}
    \item $(\cG,\k)$ has \emph{product-form} stationary distribution if there exist functions indexed by species $f_i: \N\to \R_{>0}; S_i\in \cS$ such that for any positive recurrent irreducible component $\Ga$ the stationary distribution has the form
$$\pi_\Ga(x)=Z^{-1}_\Ga\prod_{S_i\in\cS}^{ } f_i(x_i)$$
 for any $x\in\Ga$, where $Z_\Ga$ is a normalizing constant defined by\\
$Z_\Ga=\sum_{x\in\Ga}\prod_{S_i\in\cS}^{ } f_i(x_i)$.
\item $(\cG,\k)$ has \emph{$\I$-product-form} stationary distribution if there is a nontrivial\footnote{i.e $\cS\not\in\I$, and w.l.o.g. assume that $I_j\neq I_l$ for  $j\neq l$, $j,l\in [k]$} covering $\I=(I_j)_{j\in [k]}$ of the species set $\cS$ (i.e. $\cS=\cup_{j\in [k]}I_j$) with functions $f_{I_j}:\Z_{\geq0}^{I_j}\to \R_{>0}$ such that for any positive recurrent irreducible component $\Ga$ the corresponding stationary distribution has the form
$$\pi_\Ga(x)=Z^{-1}_\Ga\prod_{j=1}^{k} f_{I_j}(x_{I_j})$$
where $Z_\Ga$ is the normalizing constant.
\end{enumerate}
\end{definition}

Note that product-form corresponds to $\mathcal{I}$-product-form with $\mathcal{I}$ all 1-element sets as partition. Furthermore, product/$\I$-product-form stationary distributions only matter on nontrivial positive recurrent irreducible components (consider e.g.
$
A\to B$ for an example of a trivial one). Note also that for CRNs with infinite irreducible components, positive recurrence can be hard to check, cf., e.g., \cite{xu2019dynamics}.
\begin{remark}
Observe that Definition \ref{prod_f} is more restrictive than, e.g., only asking to have a factorisation for one positive recurrent irreducible component.
\end{remark}

For $\I$-product-form, the following holds concerning their comparison.

\begin{lemma}\label{prod_str}
For two covers of $\cS$, if $\I_1\leq\I_2$, i.e. $\I_1$ refines $\I_2$ as a cover\footnote{say $\I_1\leq\I_2$ if $\forall I_j\in\I_1\exists I'_j\in\I_2$ such that $I_j\subseteq I'_j$}, then $\I_1$ product-form implies $\I_2$ product-form for $(\cG,\k)$.
\end{lemma}

The notion of product-form is used in numerous works in applied probability  \cite{liggett,Goutsias_ov,serfozo,kipnis1999sli} and is the strongest notion, i.e. it implies any other $\I$-product-form by Lemma \ref{prod_str}. The notion of pair-product-form
stationary distribution from statistical physics (see, e.g. \cite[Section 6.1.1]{Evans1} or \cite[Section 3.6.1]{Schadschn_cx}) is an example
of $\I$-product-form,  where $\I$ consists of pairs of overlapping 2-sets. 
While it is rare not to have $\I$-product-form in studied examples, they exist.
For a mass transport model with explicit stationary distributions that is similar, but not expressible as $\I$-product-form we refer to \cite{Guioth}.

In the following we mostly focus on almost-essential conservative CRNs. By finiteness (see Remark \ref{rmk_pos_cons}), CTMC dynamics are positive recurrent on irreducible components.
\begin{definition}\label{three_cas}
We distinguish three cases for a CRN $\cG$ with given kinetics:
\begin{enumerate}
\item[(I)] $\cG$ has \emph{product-form stationary distribution independently 
of the rate}
if for any rate $\k\in\R_{>0}^\cR$, $(\cG,\k)$ has product-form stationary distribution.
\item[(N)] $\cG$ has \emph{non-product-form stationary distribution independently of the rate} if for any rate $\k\in\R_{>0}^\cR$, $(\cG,\k)$ has no product-form stationary distribution.
\item[(E)] $\cG$ can attain \emph{both product- and non-product-form stationary distribution} if that holds for different $(\cG,\k)$ with rates in $\k\in\R_{>0}^\cR$.

\end{enumerate}
\end{definition}

\subsection{Known results on stationary distributions} \label{res_stat}
The classification of the stationary behaviour of reaction networks
is challenging, and often studied via simulations \cite{gillespieb}.
In particular, analytical solutions for the stationary distribution (in case of existence) are not known for most systems. Some stationary distributions of weakly reversible reaction networks are well-understood. This is the case for complex balanced and autocatalytic CRNs.

Complex balanced CRNs \cite{anderson2} have a Poisson product-form stationary distribution \cite{anderson2} and can be characterized by that \cite{Cappelletti}: For a complex balanced CRN $(\cG,\k)$ and an irreducible component $\Ga$, the corresponding CTMC has product-form stationary distribution of the form
 $$\pi_\Ga(x)=Z_\Ga^{-1} \frac{c^x}{x!},$$
 where $c\in\R^n_{>0}$ is a point of complex balance and $Z_\Ga$ is a normalizing constant. In particular, deficiency zero weakly reversible CRN are complex balanced independently of the rate, thus they have product-form stationary distribution independently of the rate. 

While it is currently not known which weakly reversible CRN admit a product-form stationary distribution beyond complex balance, at least some more have this property. In particular, many weakly reversible non deficiency zero CRN have \\
product-form stationary distribution independently of the rate \cite{hoesslyuni}. However, not all weakly-reversible CRNs have product-form stationary distribution independently of the rate (see Section \ref{sect_ex}).

Furthermore so-called autocatalytic CRNs (\cite[Section 3]{hoesslysta}), a class of non-weakly-reversible CRNs, also have product-form stationary distributions. Their product form functions come from an infinite family of functions, where the first one specializes to the Poisson form as above.  So, for
an autocatalytic CRN in the sense of \cite[Section 3]{hoesslysta},
the stochastic dynamics
has product-form stationary distribution
\begin{equation}\label{MainFormula}
\pi_\Ga(x)=Z_\Ga^{-1}\prod_{S_i\in\cS}^{ }f_i(x_i) ,
\end{equation}
with product-form functions
$$f_i(x_i)=\lambda_i^{x_i}\frac{1}{x_i!}\prod_{l=1}^{x_i}(1+\sum_{k=2}^{n_i}\be_i^k \prod_{r=1}^{k-1}(l-r))$$ 
on its irreducible components, where $\lambda_i$ and $\be_i^k$ are determined by the autocatalytic CRN, cf. \cite[Section 3]{hoesslysta} and where $Z_\Ga$ is the normalising constant. 

Results on systematical derivation of sufficient conditions for product-form stationary distributions through decompositions of CRNs with examples can be found in \cite{hoesslyuni}. Furthermore some results on stationary distributions for CRNs with so-called discreteness-induced transitions are available \cite{bibbona_stat}. However the corresponding stationary distributions are not of product form. 
 
\subsection{Examples}
We end this section with some illustrative examples to introduce the reader further into the setting. Note that both reversible, weakly reversible and autocatalytic CRNs are essential, cf., e.g. \cite{hoesslyuni}. 
  \begin{example}\label{ex1}
 Consider the following non-weakly-reversible CRN
 $$A\xrightarrow{\be} B,\quad 2A\xleftarrow{\a} 2B,$$
which is almost essential, with stochiometric compatibility classes dividing the state space into parts $\Ga_j=\{x\in \N^2| x_1+x_2=j\}$. The $\Ga_j$ are positive irreducible components for $j\geq2$.
 \end{example}
 
 \begin{example}\label{def1}
 The following weakly-reversible deficiency one CRN
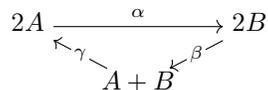
\begin{center}
\begin{tikzpicture}[description/.style={fill=white,inner sep=1pt}]
    \matrix (m) [matrix of math nodes, row sep=1em,
    column sep=1.5em, text height=1.ex, text depth=0.25ex]
    { 2A & & 2B \\
& A+B & \\ };
   
       \path[->,font=\scriptsize]
    (m-1-1) edge node[auto] {$ \a $} (m-1-3)
          (m-2-2)  edge node[description] {$ \gamma $} (m-1-1)
    (m-1-3) edge node[description] {$\be $} (m-2-2);
\end{tikzpicture}
\end{center}

is almost essential with state space decomposition $\Ga_j=\{x\in \N^2| x_1+x_2=j\} $, where the $\Ga_j$ are positive for $j\geq2$. This CRN is complex balanced if and only if $\gamma^2=\a\be$.
 \end{example}

\section{Geometric view on product-form stationary distributions}\label{geom_v}

In this section we study algebraic and geometric aspects of the solutions to the Master equation $ME(\cG,\Ga)$ on an irreducible component for a given network $\cG$.
In Section \ref{h} we express the equations in $ME(\cG,\Ga)$ in matrix form, which allows us to express $\pi _{\Gamma}$ in terms of the reaction rates $\k$. In Section \ref{app_prod_fn} we derive conditions on the values of $\pi _{\Gamma }(x)$ for each $x\in  \Gamma $ that allow $(\mathcal{G} ,\k)$ to have product-form stationary distribution, or in some cases guarantee it. These conditions will give rise to polynomial ideals. In Section \ref{semi}, we explore the algebraic-geometric connection between these ideals of polynomial conditions and the sets of reals or positive values satisfying them. From this connection, we extract a description for the set of reaction rates that give product-form stationary distributions for $(\mathcal{G},\k)$ under certain hypotheses. The basic terminology and ideas needed to approach this connection are given in Section \ref{sec:alg_geom}

\subsection{Algebraic and geometric tools}\label{sec:alg_geom}

For simplicity we will focus on the concrete case of polynomials with real coefficients and their real solutions. 
When considering solutions to polynomial equations, it is often useful to consider all equations at the same time. Polynomial ideals formalise this idea, as they are subsets of a polynomial ring with real coefficients (like $\mathbb{R}[y_1,\ldots ,y_s]$ where $y_1,\ldots ,y_s$ are the variables), which are closed by nonlinear combinations of its elements with coefficients in the ring. For more details we refer e.g. to \cite{Cox}. 
Polynomial ideals have the following property, which follows e.g. from Hilbert's Basis Theorem (cf., e.g., \cite[Section 2.5]{Cox}):
\begin{lemma}\label{Lemma:noet}
Let $R:=\mathbb{R}[y_1,\ldots ,y_s]$ be any polynomial ring. Then the two following  equivalent properties hold:
\begin{itemize}
	\item Any ideal $I\subset R$ is generated by finitely many elements.
	\item Any increasing chain of ideals $I_0\subseteq I_1\subseteq \ldots \subseteq I_i\subseteq \ldots  $ eventually stabilizes. That is, there exists $N\in \mathbb{N}$ such that for any $j>N$, $I_N=I_j$.
\end{itemize}
\end{lemma}

An \emph{algebraic set} is the set of common solutions to a finite collection of polynomial equations. A \emph{semi-algebraic set} is the set of common solutions to a finite collection of polynomial equations and polynomial inequalities.

As mentioned above, the idea of the connection between polynomial ideals and geometric sets is that any solution of a set of polynomial equations is also a solution of any non-linear combination of these equations. This establishes a connection between ideals and algebraic sets. 

Given an ideal $I\subset \mathbb{R}[y_1,\ldots ,y_s]$, we commonly denote by $\mathbb{V}(I)$ the algebraic set
$$\mathbb{V}(I)=\{ x\in \mathbb{R}^s:f(x)=0,\, \forall f\in I\}\subset \mathbb{R}^s\mbox{.}$$
If $I=\langle f_1,\ldots ,f_d\rangle \subset \mathbb{R}[y_1,\ldots ,y_s]$ is the polynomial ideal generated by $f_1,\ldots ,f_d$, then 
$$\mathbb{V}(I)=\cap _{i=1}^{d}{\mathbb{V}(\langle f_i\rangle )}\mbox{.}$$
As $I\subseteq J$ means that $J$ imposes at least the same conditions as $I$, then
$$I\subseteq J\Longrightarrow \mathbb{V}(I)\supseteq\mathbb{V}(J)\mbox{.}$$
Furthermore $\mathbb{R}(y_1,\ldots ,y_s)$ will denote the field of rational functions over the variables $y_1,\ldots ,y_s$, which is given by
$$\mathbb{R}(y_1,\ldots ,y_s):=\{\ \frac{f}{g}\st f,g\in\mathbb{R}[y_1,\ldots ,y_s],g\not = 0\} $$

\subsection{Expressing the solution set as a kernel}\label{h}\hfill\break
Let $\mathcal{G}$ be a reaction networks with CTMC dynamics having finite irreducible components. Recall that $ME(\cG,\Ga)$ stands for the set of equations derived from the Master equation \eqref{master_eq1}, for each $x\in \Gamma $. For a fixed value of the reaction rate vector $\k\in\R^\cR_{>0}$, the solution to this set of equations is a 1-dimensional subspace of $\R^\Ga$ (see Remark \ref{rmk_pos_cons}) and the stationary distribution is the solution in $\Delta_{\Ga}\subset \R^\Ga$. To formalize this, we construct the following map:
\begin{align}\label{eq:map_h_G}&\R^\cR_{>0}\xrightarrow{h_\Ga}\R^\Ga_{>0}\xrightarrow{ p }\R\Pr^\Ga_{>0}\simeq\Delta_{\Ga}\\
&\k\mapsto h_\Ga(\k)\mapsto \overline{h_\Ga(\k)}\mapsto \frac{1}{Z}h_\Ga(\k):=(\pi_\Ga(x):x\in\Ga)\mbox{.}\nonumber \end{align}
We will show that the map $h_\Ga$ is polynomial and unique (modulo normalisation, cf. Corollary \ref{clm2}). Note that $p$ is simply the projection into real projective space, and the last map is the normalisation to a probability distribution (i.e. the well-known isomorphism $\R\Pr^\Ga_{\geq 0}\simeq \Delta_{\Ga}$). For more on basic projective algebraic geometry  we refer to, e.g. \cite[Section 8]{Cox}.

Suppose that the following holds for the kinetics:
\begin{assumption}\label{ass1}
The transition intensity functions $\la_{\nu \to \nu\p}(\cdot)$ are linear polynomials in the reaction rate parameters $\k_{\nu \to \nu\p}$ for all reactions $\nu \to \nu\p\in\cR$. 
\end{assumption}
As an example of transition intensity functions satisfying Assumption \ref{ass1}, consider Mass-action kinetics as given in equation \eqref{int}, and the CTMC on a non-trivial irreducible component $\Ga$.
Then $ME(\cG,\Ga)$ is the system of equations given by
$$
  \sum_{\nu_i\to\nu_i\p\in\cR} \pi_\Ga(x+\nu_i-\nu_i\p)\la_i(x+\nu_i -\nu_i\p)=\pi_\Ga(x)\sum_{\nu_i\to\nu_i\p\in\cR}\la_i(x) ,\qquad x\in \Gamma .
$$

Note that the left hand side is not zero for any $x$ in $\Ga$. Taking everything to the left hand side we can write $ME(\cG,\Ga)$ as
$$\sum_{\nu_i\to\nu_i\p\in\cR} \pi_\Ga(x+\nu_i-\nu_i\p)\la_i(x+\nu_i -\nu_i\p)-\pi_\Ga(x)\sum_{\nu_i\to\nu_i\p\in\cR}\la_i(x)=0,\qquad x\in \Gamma \mbox{.}$$
We now express $ME(\cG,\Ga)$ in matrix form. Let $\Ga=\{c_1,c_2,\cdots, c_m\}$ be a (finite) arbitrary irreducible component with $m$ elements (see Example \ref{ex1} and \ref{ex1_} for examples of irreducible components of a CRN).
Then the system of equations $ME(\cG,\Ga)$ can be written as:\hfill\break
\begin{equation}
\begin{tikzpicture}[baseline=(current bounding box.center)]
\matrix (m) [matrix of math nodes,nodes in empty cells,right delimiter={]},left delimiter={[} ]{
-\sum_{\nu_i\to\nu_i\p\in\cR}\la_i(c_1)  &  &   &  &  &   \\
  & & & & &  \\
 & & & & &    \\
   & & & & &   \\
  & & & & &  \\
 & & &  &  & -\sum_{\nu_i\to\nu_i\p\in\cR}\la_i(c_m)\\
} ;
\draw[loosely dotted] (m-1-1)-- (m-6-6);
\draw[loosely dotted] (m-1-1)-- (m-1-6);
\draw[loosely dotted] (m-1-1)-- (m-6-1);
\end{tikzpicture} 
\begin{pmatrix}
    \pi_\Ga(c_1)\\\pi_\Ga(c_2)\\\vdots\\\pi_\Ga(c_m)
\end{pmatrix}=0
\label{eq:eqq1}
\end{equation}
where the entries outside of the diagonal are of the form $\sum_{c_k+\nu_i -\nu_i\p=c_j}\la_i(c_k+\nu_i -\nu_i\p)$ in the k-row and j-column for $j\neq k$. Note that, under Assumption \ref{ass1}, each $\la_i(c_k+\nu_i -\nu_i\p)$ is a homogeneous polynomial of degree $1$ in the rates $\k_i$.

We denote the corresponding matrix by $A(\k)$, so that \eqref{eq:eqq1} becomes $A(\k)\overline{\pi_{\Gamma}}=0$, where $\overline{\pi_{\Gamma}}=(\pi_\Ga(c_1),\ldots \pi_\Ga(c_m))$. See Example \ref{ex1_} for an example of the computation of $A(\k)$.
 \begin{lemma}\label{clm}  Consider $A(\k)$ as a matrix with entries in the field of fractions $\R(\k_i|{\nu_i\to\nu_i\p}\in\cR)$. Under Assumption \ref{ass1}, $\mathrm{rank}(A(\k))=m-1$\end{lemma}
\begin{proof} For any $a\in \R^\cR_{> 0}$, the kernel of $A(a)$ is 1-dimensional by Markov chain theory (see Remark \ref{rmk_pos_cons}). The determinant $det(A(\k))$ is an element of $\R[\k_i|{\nu_i\to\nu_i\p}\in\cR]$,  which vanishes for all $a\in \R^\cR_{> 0}$. So it is identically zero as a polynomial in $\R[\k_i|{\nu_i\to\nu_i\p}\in\cR]$.\hfill\break
For any $a\in \R^\cR_{> 0}$ there is some $m-1$-minor of $A(a)$ which does not vanish. Hence some $m-1$-minors of $A(\k)$ are nonzero,
and we conclude that $A(\k)$ has rank $m-1$ over $\R(\k_i|{\nu_i\to\nu_i\p}\in\cR)$.
\end{proof}

\begin{corollary}\label{clm2}
Under Assumption \ref{ass1}, the map $h_\Ga$ can be given as a unique vector $(F_1,\cdots , F_m)$, up to multiplication by a constant in $\R^*$, whose entries are homogeneous polynomials in $\R[\k_i|{\nu_i\to\nu_i\p}\in\cR]$ of the same degree, with positive coefficients and no non-constant common divisor.  Hence the $F_i$ never vanish on $\R\Pr_{>0}^\cR$. 
\end{corollary}

\begin{proof}
For the ease of the reader we complete the computation of $h_\Ga$ by recalling the computation of the kernel of $A(\k)$ (which is one-dimensional by Lemma \ref{clm}) via Gaussian elimination over the field of fractions $\R(\k_i|{\nu_i\to\nu_i\p}\in\cR)$. This gives a constructive proof for $h_\Ga$ and $(F_1,\cdots,F_m)$. We can proceed as follows:
$$(0)\quad \left[
    	\begin{array}{cc|cc}
    	A(\k) \\
    	\hline
    	I_m \\	
    	\end{array}
    	\right]\xRightarrow{(1)\&(2)}
	\left[
    	\begin{array}{cc|cc}
    	B \\
    	\hline
    	C \\	
    	\end{array}
    	\right]\quad (3),	$$

\begin{enumerate}
\setcounter{enumi}{-1}
\item We start by creating the row-augmented matrix, where $I_m$ is the $m\times m$ identity matrix.
\item We divide each j-column by $-a_{j,j}=\sum_{\nu_i\to\nu_i\p\in\cR}\la_i(x_j)$. Hence we get $-1$-entries in the diagonal of the former matrix $A(\k)$. 
\item We apply elementary column operations to get the upper matrix in column echelon form. Note that we only add columns multiplied by elements in $\R(\k_i|{\nu_i\to\nu_i\p}\in\cR)$ with all coefficients positive.
\item We get the matrices $B,C$, where $B$ is in column echelon form of degree $m-1$. Furthermore  matrix $C$ has the form $$C=\begin{pmatrix}
c_{1,1}\,\cdots\, c_{1,m-1}\, v_1\\
\vdots\,\ddots\,\quad\,\vdots\,\,\,\,\,\,\,\,\,\,\,\,\,\,\,\,\vdots\\
c_{m,1}\cdots\, c_{m,m-1}\, v_m
\end{pmatrix}$$
	such that the last column of $C$, i.e. $(v_1,\cdots, v_m)$, is a non-trivial vector of the kernel.
\end{enumerate}
Hence $(v_1,\cdots, v_m)$ is a basis of the kernel over $\R(\k_i|{\nu_i\to\nu_i\p}\in\cR)$, and we can choose the corresponding column vector as homogeneous polynomials. If the greatest common divisor(GCD) in $\R[\k_i|{\nu_i\to\nu_i\p}\in\cR]$ is nontrivial we divide all polynomials of the column vector by the GCD. Denote the corresponding vector by $w$ with entries in $\R[\k_i|{\nu_i\to\nu_i\p}\in\cR]$, which is unique up to multiplication by a constant in $\R^*$.
 Then the vector $w$ has the form $w=(F_1,\cdots ,F_m)$, and by construction the $F_1,\cdots ,F_m$ are homogeneous polynomials of the same degree with positive coefficients by (2). As all the coefficients of the polynomials are positive, it is clear that $F_1,\cdots, F_m$ never vanish on $\R\Pr_{>0}^\cR$.
\end{proof}

Using the former construction, we can define a map $h_{\Gamma}^*$ between two polynomial rings, namely the ring of polynomials whose variables are $\pi (x)$, $x\in \Gamma$, and the ring containing polynomials with variables $\k _i$ encoding the reaction rates:
$$\mathbb{R}[\pi (x):x\in \Gamma]\stackrel{h_{\Gamma }^*}{\longrightarrow}\mathbb{R}[\k_i|{\nu_i\to\nu_i\p}\in\cR].$$  
This map sends $\pi(x_j)$ to $F_j(\k)$, and is a ring homomorphism (a map preserving addition and multiplication). Its geometric counterpart is the following map of algebraic sets:
$$\k\in \R^\cR\longmapsto h_{\Gamma }(\k)=(F_1(\k),\cdots,F_m(\k))\in \R^\Ga$$
 which extends the first part of the map \eqref{eq:map_h_G} to $\R^\cR$. We use the same name for this new map as for its restriction to $\R^\cR_{>0}$. By Corollary \ref{clm2}, the image of $\R_{>0}^\cR$ by this map is contained in $\R_{>0}^\Ga$.

  We will use this map later in this section.

\vspace{0.2cm}

From now on we allow the following abuse of notations. We will write $\pi$ for the stationary distribution without specifying the irreducible component, $\pi_\Ga$ when specifying it to $\Ga$, $\pi_x$ for the value of the stationary distribution on $\Ga$ with coordinate $x\in \Ga$, $h_\Ga(\k)$ for the vector of polynomials from Corollary \ref{clm2} on $\Ga$ and $h_x(\k)$ for the corresponding polynomial of $h_\Ga(\k)$ in coordinate $x\in \Ga$. Furthermore, if the context is clear and we consider a sequence of irreducible components $\Ga_l$ for $l\in \I$ along $\I$ an index set, we might write $\pi_l$ and $h_l(\k)$ for $\pi_{\Ga_l}$ resp. $h_{\Ga_l}(\k)$. 

\begin{example}\label{ex1_}
Consider Example \ref{ex1} with Mass-action kinetics. The Master equation is given as
$$\a\pi(x_1-2,x_2+2)(x_2+2)(x_2+1)\1_{\{x_1\geq 2\}}+\be\pi(x_1+1,x_2-1)(x_1+1)\1_{\{x_2\geq 1\}}$$ 
$$=\pi(x_1,x_2)(\a x_2(x_2-1)\1_{\{x_2\geq 2\}}+\be x_1\1_{\{x_1\geq 1\}})$$
We restrict our treatment to the first three positive irreducible components $\Ga_2,\Ga_3,\Ga_4$, where we write the defining equations in matrix form:

\begin{itemize}

\item $ME(\cG,\Ga_2)$ is defined by:\hfill\break
\begin{equation*}
\begin{pmatrix}
-2\be  &0  &   2\a   \\
  2\be& -\be& 0  \\
0&\be&-2\a    
\end{pmatrix}
\begin{pmatrix}
    \pi_{20}\\\pi_{11}\\\pi_{02}
\end{pmatrix}=0
\end{equation*}

\item
$ME(\cG,\Ga_3)$ is defined by:\hfill\break
\begin{equation*}
\begin{pmatrix}
-3\be  &0  &   2\a &0  \\
  3\be& -2\be& 0&6\a  \\
0&2\be&-(2\a  +\be)&0\\
0&0&\be&-6\a 
\end{pmatrix}
\begin{pmatrix}
    \pi_{30}\\\pi_{21}\\\pi_{12}\\\pi_{03}
\end{pmatrix}=0
\end{equation*}

\item
$ME(\cG,\Ga_4)$ is defined by:\hfill\break
\begin{equation*}
\begin{pmatrix}
-4\be  &0  &   2\a &0&0  \\
  4\be& -3\be& 0&6\a &0 \\
0&3\be&-(2\a  +2\be)&0&12\a\\
0&0&2\be&-(6\a+\be)&0\\
 0&0&0&\be&-12\be&\\
\end{pmatrix}
\begin{pmatrix}
    \pi_{40}\\\pi_{31}\\\pi_{22}\\\pi_{13}\\\pi_{04}
\end{pmatrix}=0
\end{equation*}

\end{itemize}
By Lemma \ref{clm} the kernels $A(\k)$ over $\R(\k_i|{\nu_i\to\nu_i\p}\in\cR)$ are one-dimensional. They can be computed, e.g. using Maple, cf. Section \ref{subs_examples}. The entries of a vector in $\mathrm{ker}(A(\k))$ are polynomials in $\R[\k_i|{\nu_i\to\nu_i\p}\in\cR]$ with positive coefficients and, possibly after dividing by the GCD (as in the proof of Corollary \ref{clm2}), they correspond to the entries of the polynomial map $h_{\Ga_l}$, as computed in Corollary \ref{clm2}. In this case, these vectors are:
\begin{itemize}
\item
$h_{\Ga_2}(\a,\be)=(\a,2\a,\be)$
\item $h_{\Ga_3}(\a,\be)=(4\a^2,(2\a+\be)3\a,6\a\be,\be^2)$
\item $h_{\Ga_4}(\a,\be)=(3\a^2(6\a+\be),24\a^3+28\a^2\be,6\a\be(6\a+\be),12\a\be^2,\be^3)$
\end{itemize}

\end{example}

\subsection{Algebraic relations for product-form stationary distributions}\label{app_prod_fn}
Consider an indexed set of pairwise disjoint subsets $\Ga_l$ of $\Z_{\geq 0}^n$ corresponding to irreducible components with stationary distributions $\pi_l\in \mathring\Delta_{\Ga_l}$ coming from a CRN $(\mathcal{G}, \k)$. According to Definition \ref{prod_f}, $(\mathcal{G}, \k)$ with (indexed) positive recurrent irreducible components and stationary distributions $(\Ga_l,\pi_l)_{l\in \mathcal{I}}$ has product-form stationary distribution if there are functions $f_i: \N\to \R_{>0}; S_i\in \cS$ such that for any subset $\Ga_l$ the distribution has the form
$$\pi_{\Ga_l}(x)=Z_{\Ga_l}^{-1}\prod_{S_i\in\cS}^{ } f_i(x_i),\quad Z_{\Ga_l}=\sum_{x\in \Ga_l}\prod_{S_i\in\cS}^{ } f_i(x_i).$$ 

The next result follows from combining Lemma \ref{Lemma_nec} and Proposition \ref{suff_cond}, which are given in Section \ref{app_prod_f}:
\begin{theorem}\label{suff_condn}  Let $(\mathcal{G}, \k)$ be a CRN with positive recurrent irreducible components and stationary distributions $(\Ga_l,\pi_l)_{l\in \mathcal{I}}$ with $\pi_l\in \mathring\Delta_{\Ga_l}$, where\\
$\Ga_l=\{x\in \Z^\cS_{\geq 0}|\sum_{S_i\in \cS} x_i=l\}$ with index set 
$\mathcal{I}=\Z_{\geq 1}$. Then the following are equivalent:
\begin{enumerate}
\item $(\mathcal{G}, \k)$ has product-form
\item The following non-linear relations are satisfied by the distributions:\\
(a) For all $x,x+e_j-e_k\in \Ga_i$ and $y,y+e_j-e_k\in \Ga_{i+1}$, $i\in \mathcal{I}$  with $x_j=y_j,x_k=y_k$ we have
$$\pi_{i}(x+e_j-e_k)\pi_{i+1}(y)=\pi_{i+1}(y+e_j-e_k)\pi_i(x)   $$
(b) For all $x,z\in \Ga_i;\quad x+e_j,y,z+e_k,w\in \Ga_{i+1}$ and $y+e_j,w+e_k\in \Ga_{i+2}$, $i\in \mathcal{I}$, with $x_j=y_j$ and $z_k=w_k$ we have
$$\pi_{i+1}(x+e_j)\pi_{i+1}(y)\pi_{i+2}(w+e_k)\pi_i(z)=\pi_{i+1}(z+e_k)\pi_{i+1}(w)\pi_{i+2}(y+e_j)\pi_i(x)$$
\end{enumerate}
 Moreover, if $\mathcal{I}=\Z_{\geq \intval}$ for some $\intval\geq 2$, then $(1)\Rightarrow (2)$ still holds. 
\end{theorem}

 Theorem \ref{suff_condn} gives equivalent conditions for the sequence of probability distributions $(\pi _{\Gamma _l}:l\in \mathcal{I})$ to have product-form. These conditions are expressed as relations between the values of the distribution on different elements of the disjoint subsets of $\Z_{\geq 0}^n$. However, only under strict assumptions on the decomposition of $\Z_{\geq 0}^n$ the conditions from (2) are equivalent to product-form. For more general decompositions of the state space (i.e. into irreducible components), analogous algebraic relations often exist, but they give only necessary conditions for product-form stationary distributions. Likewise, algebraic conditions for $\I$-product-form stationary distributions might be worked out.

\subsection{The semi-algebraic set of product-form stationary distributions}\label{semi}
In this part we use the compatibility conditions for product-form stationary distributions from Section \ref{app_prod_fn} in combination with the homogeneous polynomials from the polynomial vectors $h_{\Ga_i}(\k)$ in Section \ref{h} representing the stationary distributions. The fist are conditions on the values taken by $\pi _{\Gamma}(x)$ at each $x\in\Gamma$, while the latter are expressions for these $\pi _{\Gamma }(x)$ in terms of the $\k$. Combining these leads to algebraic conditions expressing the compatibility of product-form stationary distribution, now  on the reaction rates. Correspondingly, we define an ideal expressing all these compatibility conditions, together with the set of possible values of the rate parameters which satisfy these conditions. This set will be a semi-algebraic set (see Section \ref{sec:alg_geom}), given by the intersection of the positive orthant of the rate space $\R^\cR_{>0}$ and the real algebraic set defined by the real zeros of the mentioned ideals.

Consider a CRN $\cG$ with given kinetics under Assumption \ref{ass1}, and the indexed set $(\Ga_l,h_l(\k))_{l\in \mathcal{I}}$ given by the decomposition of the state space into irreducible components $\Ga_l$ with unnormalised stationary distributions $h_l(\k)$ (parametrised by polynomials, see Section \ref{h}). We assume now that the irreducible components are of the form
$\Ga_l=\{x\in \Z^\cS_{\geq 0}|\sum_{S_i\in \cS} x_i=l\}$, with  $\mathcal{I}=\Z_{\geq \intval}$. In other words, the irreducible components consist of the integer points of $l$-dilations of the simplex for all $l\geq\intval$, where $\intval\in\Z_{\geq 1}$ is an arbitrary positive integer (that depends on the CRN under consideration).   
Further, recall that the $h_l(\k)$ are polynomials in the variables $\k_i$ as proven in Proposition \ref{clm2}. 
\begin{definition} \label{product_ideal}\hfill\break Consider for each $j\in \mathcal{I}$ the ideal $J_j\subset \R[\pi _{\Ga_i}(x):x\in \Ga_i, \intval\leq i\leq j]$ defined as
\small
\begin{align*}
J_j=&\langle \text{homogeneous polynomials from Theorem \ref{suff_condn} (a), (b) for }\Ga_i:\intval\leq i\leq j,\rangle \mbox{,}\end{align*}
and the ideal

$$J=\left\langle \text{all homogeneous polynomials from Theorem \ref{suff_condn} (a), (b) for all } i \geq \intval\right\rangle$$

\normalsize
We denote by $I_j$, for $j\in \mathcal{I}$, and $I$, respectively, their images in $\R[\k_i|\nu_i\to \nu_i\p\in\cR ]$ via $h_{\Ga }^*$ . 
Correspondingly we have the following zero sets:\hfill\break
- The affine algebraic set: 
$$V_\cG:=\mathbb{V}(I)=\{\k\in \R^\cR|f(\k)=0 \, \forall f\in I\}$$  
- The affine semi-algebraic set:
$$V_{\cG,>0}:=\mathbb{V}(I)_{>0}=\mathbb{V}(I)\cap\R^\cR_{> 0}$$
\end{definition}
\begin{remark}
Note that by construction the corresponding ideals $I_j,I$ are then simply generated by the relations from Theorem \ref{suff_condn} (a), (b) but with $\pi _{\Ga_i}(x)$ replaced by the polynomial in the $x$-coordinate of the vector from the map $h_{\Ga_i}(\k)$.
\end{remark}

\begin{remark}\label{suff_rem_} 
\item We have $I_j\subseteq I_{j+1}\subseteq I$, hence
  $\mathbb{V}(I)\subseteq \mathbb{V}(I_{j+1})\subseteq \mathbb{V}(I_{j})$ for all  $j\geq \intval$. Similarly, $J_j\subseteq J_{j+1}\subseteq J$, so
  $\mathbb{V}(J)\subseteq \mathbb{V}(J_{j+1})\subseteq \mathbb{V}(J_{j})$ for all  $j\geq \intval$.
	\end{remark}
	
	\begin{remark}
\item The ideal $I$ is  a polynomial ideal in the ring $\R[\k_i|\nu_i\to \nu_i\p\in\cR ]$, hence $I$ has a finite set of generators (see Lemma \ref{Lemma:noet}). Equivalently, there exists some $N\geq \intval$ such that $I_{\intval}\subseteq \ldots \subseteq I_N=I_{N+1}=\ldots =I$.

\end{remark}

 Combining Definition \ref{product_ideal} with Theorem \ref{suff_condn} gives the following result:
\begin{theorem}\label{semi_alg}
 Let $\cG$ be a conservative CRN with kinetics satisfying Assumption \ref{ass1} and with irreducible components of the form
$\Ga_l=\{x\in \Z^\cS_{\geq 0}|\sum_{S_i\in \cS} x_i=l\}$ for the index set $\mathcal{I}=\Z_{\geq 1}$. The following statements are equivalent for any $a\in \R^\cR_{>0}$
\begin{enumerate}
    \item $(\cG,a)$ has product form stationary distribution
    \item $a\in \mathbb{V}(I)$
\end{enumerate}
Moreover, if $\mathcal{I}=\Z_{\geq \intval}$ for some $\intval\geq 2$, then $(1)\Rightarrow (2)$ still holds. 

\begin{proof}
$(1)\implies (2)$: If $a\in \R^\cR_{>0}$ is such that $(\cG,a)$ has product-form stationary distribution, any relation fromLemma \ref{Lemma_nec} (2) is satisfied by the images of the parametrised stationary distributions via $h_\Ga$ by definition. Hence $a\in \mathbb{V}(I)$. Note that if $\mathcal{I}=\Z_{\geq \intval}$ for some $\intval\geq 2$, then the implication remains valid as the corresponding implication from Lemma \ref{Lemma_nec} (1) $\implies $ (2) still holds.\hfill\break
$(2)\implies (1)$ follows similarly from the corresponding implication of Lemma \ref{Lemma_nec}
\end{proof}
\end{theorem}

\begin{example} We continue Examples \ref{ex1} and \ref{ex1_}. To check whether $\mathbb{V}(I)_{>0}$ is non-empty amounts by construction to check whether for all $j\geq 1$, $\mathbb{V}(I_j)_{>0}$ is nonempty (by Remark \ref{suff_rem_}).  However, Remark \ref{suff_rem_} shows that there is a finite set of polynomial conditions that one can check, even thought it might not be easy to find. Let us point out first that there are no relations coming from Theorem \ref{suff_condn} (a), so we only have to check the relations from (b). Moreover note that even if all relations in the ideal must be checked to guarantee product-form, finding an element with no positive solutions in this ideal already proves that there is no product-form independently of the rates.

 Consider the following element of $J_4\subseteq J$
 $$
 \pi_{\Ga_3}(0,3)\pi_{\Ga_3}(1,2)\pi_{\Ga_2}(1,1)\pi_{\Ga_4}(2,2)-\pi_{\Ga_3}(2,1)\pi_{\Ga_3}(1,2)\pi_{\Ga_2}(0,2)\pi_{\Ga_4}(1,3)
 $$
  from Definition \ref{product_ideal} (where we are using the polynomials with variables in coordinates of irreducible components, see Example \ref{ex1_}). Then the image in $\R[\k_i|\nu_i\to \nu_i\p\in\cR ]$ is given by

  $$h_{(0,3)}^{}(\k)h_{(1,2)}^{}(\k)h_{(1,1)}^{}(\k)h_{(2,2)}^{}(\k)-h_{(2,1)}^{}(\k)h_{(1,2)}^{}(\k)h_{(0,2)}^{}(\k)h_{(1,3)}^{}(\k)\mbox{,}$$  
were we write $h_{x}^{}(\k)$ for the polynomial 
from Section \ref{h} that we constructed in Corollary \ref{clm2}. Then, inserting the homogeneous polynomials $h_\Ga(\k)$ from Example \ref{ex1_} and setting it to zero gives the following condition in the reaction rates $$6\a^3\be^4[(6\a+\be)-3(2\a+\be)]=6\a^3\be^4(-2\be)=0$$ It is easy to see that there are no positive solutions to this equation. We conclude that $\mathbb{V}(I)_{>0}$ is empty and this CRN has non-product-form stationary distribution independently of the rate. 
\end{example}

 \begin{remark}
For more general conservative or subconservative CRNs, corresponding algebraic relations as in Theorem \ref{suff_condn} usually still exist. However, they are only necessary for product-form stationary distribution, but not sufficient.
 \end{remark}

\section{Examples}\label{exa}
We consider here several classes of examples, most of which have a particular state space decomposition into irreducible components. We start with CRNs, where we first study weakly reversible CRNs, and then move beyond the case of weak reversibility. Finally we illustrate some other Markov processes from applied probability where the algebraic relations from Theorem \ref{suff_condn} apply as well.

\subsection{Reaction networks}\label{sect_ex}\hfill\break
We consider CRNs with stochastic Mass-action kinetics that have a state space decomposition along the same irreducible components as in Theorem \ref{suff_condn}. Hence on each such irreducible component its CTMC dynamics are positive recurrent.
We mainly aim to distinguish the three possibilities for the rate space concerning product-form stationary distributions (see Definition \ref{three_cas}), which we abbreviate by 
\begin{align*}
&N = \text{nonexistence of product-form stationary distribution},\\
 &I = \text{existence of product-form stationary distribution independently of the rate},\\
 &E = \text{existence of both product and non-product-form stationary distribution}.
 \end{align*}
 \begin{remark}
 As a consequence of Theorem \ref{suff_condn}, for a CRN $\cG$ with $\I=\Z_{\geq 1}$ we have the following equivalences in terms of the sets in Definition \ref{product_ideal}:
 $$N\iff V_{\cG,>0}=\emptyset,\quad I\iff V_{\cG,>0}=\R^\cR_{>0},\quad E\iff \emptyset\neq V_{\cG,>0}\subsetneq\R^\cR_{>0}$$
 \end{remark}
Consider the following CRN with reactions of molecularity up to two for future reference.
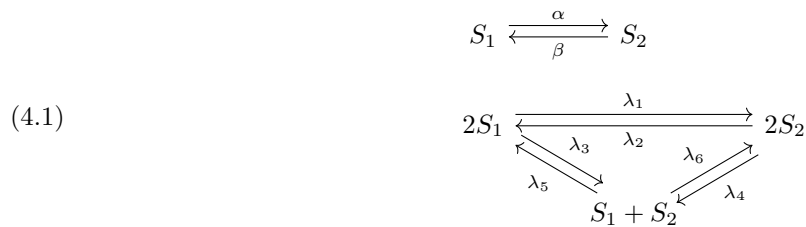
\begin{equation}\begin{tikzcd}\label{eq:big_RN}
S_1\arrow[r, shift left=1ex,"\alpha"]  &S_2 \arrow[l,"\beta"]\\ 
2S_1 \arrow[rr, shift left=1ex,"\la_1"] 
\arrow[dr,  shift left=1ex,"\la_3"] 
&& 2S_2 \arrow[dl, shift left=1ex,"\la_4"] \arrow[ll,"\la_2"]\\
& S_1+S_2\arrow[ul,"\la_5"]\arrow[ur,"\la_6"]
\end{tikzcd}
\end{equation}

\subsubsection{Examples of weakly reversible CRNs}\hfill\break
Let $\cG$ be a weakly reversible CRN. Then we know that there exist complex balanced reaction rates for which the system has product-form stationary distribution of Poisson form \cite{anderson2}. In particular, we recall that the reaction rates where the CRN is complex balanced is given by the toric variety from the moduli ideal, which we denote here by $M_\cG$ as in \cite{sturmfels}. Hence we have the following:
\begin{lemma}\label{comp_toric}
Let $\cG$ be weakly reversible with state space decomposition as in Theorem \ref{suff_condn} and with $\mathcal{I}=\Z_{\geq m}$.
Then the following relation holds between the positive part of the toric variety $V_{>0}(M_\cG)$ from \cite{sturmfels} and $V_{\cG,>0}$ from Section \ref{semi}:
$$\emptyset\neq V_{>0}(M_\cG)\subseteq V_{\cG,>0}$$
Furthermore if the deficiency of $\cG$ is zero, we have $V_{>0}(M_\cG)= V_{>0}(I_\cG)=\R^\cR_{>0}$.
\end{lemma}
As there are always reaction rate values where weakly reversible CRNs are complex balanced, they are at least of type $E$ or even of type $I$. Indeed, it is well-known that many weakly reversible CRNs are complex balanced independently of the rate (type I), for example if they have zero deficiency. Furthermore, it was shown in \cite{hoesslyuni} that many weakly reversible CRNs of higher deficiencies 
 are also of type $I$, i.e. they satisfy $V_{>0}(M_\cG)\subsetneq V_{\cG,>0}=\R^\cR_{>0}$. 

We next collect some examples and refer to Section \ref{subs_examples} for the computations.
\begin{center}

\begin{tabular}{ |c|c|c|c| }
 \hline
Index  & CRN & deficiency &classification \\

 \hline
 W1  & $S_1\ce{<=>} S_2$ &$0$& I  \\
            \hline 
W2  & $S_1\ce{<=>} S_2$ &$1$& I  \\
&$S_1+S_2\ce{<=>} 2S_2$&&\\
            \hline
 W3  & $S_1\ce{<=>} S_2$ &$1$& E  \\
 &$2S_1\ce{<=>} 2S_2$&&\\
            \hline
 W4  & $2S_1\ce{<=>} S_1+S_2$ &$1$& E  \\
 &$2S_1\ce{<=>} 2S_2$&&\\
            \hline 
W5  & $
 S_1 \ce{<=>}S_2 \ce{<=>} S_3$
  &$2$& I  \\
  &$2S_1 \ce{<=>} S_1+S_2,\quad 2S_3\ce{<=>} S_2+S_3$&&\\
            \hline
W6  & $
 S_1 \ce{<=>}S_2 \ce{<=>} S_3, \quad S_2\ce{<=>}S_4$
  &$2$& I  \\
  &$2S_1 \ce{<=>} S_1+S_2,\quad 2S_3\ce{<=>} S_2+S_3$&&\\
            \hline
\end{tabular}
\end{center}

The CRNs W1,W2, W5 and W6 all have product form stationary distributions independently of the rate. While the stationary distribution of a weakly reversible CRN is projectively equivalent to a Poisson form if and only if the reaction rate values are complex balanced \cite{Cappelletti}, some weakly reversible CRNs of nonzero deficiency have product-form stationary distributions independently of the rate, i.e. beyond complex balance (like $W2,W5,W6$ above; more can be found in \cite{hoesslyuni}). 
On the other hand, from our investigation into the above examples we observe the following (proofs can be found in Section \ref{subs_examples}):
\begin{lemma}
W3 and W4 have product-form stationary distribution if and only if they are complex balanced, i.e. for W3 and W4 $V_{>0}(M_\cG)= V_{\cG,>0}$.
\end{lemma}

This leads us to conjecture the following to be true.
\begin{conjecture}
Consider a conservative and weakly reversible CRN $\cG$ under\\ Mass-action kinetics. Assume that for each $r\in\cR$ there is no $\tilde{r}\in\cR,r\neq \tilde{r}$ such that $\nu _r\p-\nu _r=\nu _{\tilde{r}}\p-\nu _{\tilde{r}}$. Then $(\cG,\k)$ has product-form stationary distribution if and only if $(\cG,\k)$ is complex balanced.
\end{conjecture}
In particular, assuming the positive recurrence conjecture holds (see, e.g., \cite{andersonpos}) we believe the following might be true.
\begin{conjecture}
Consider a weakly reversible CRN $\cG$ under Mass-action kinetics such that for each $r\in\cR$ there is no $\tilde{r}\in\cR,r\neq \tilde{r}$ with $\nu _r\p-\nu _r=\nu _{\tilde{r}}\p-\nu _{\tilde{r}}$. Then $(\cG,\k)$ has product-form stationary distribution if and only if $(\cG,\k)$ is complex balanced, i.e. it only exhibits Poissonian product-form stationary distributions.
\end{conjecture}
Another natural question, assuming the positive recurrence conjecture holds, is \\whether the region of the reaction rate space where the CRN has product-form stationary distribution is semi-algebraic for any weakly reversible CRN under Mass-action kinetics.

\subsubsection{Examples of CRNs beyond weak reversibility}\hfill\break
Let $\cG$ be a non-weakly reversible CRN whose state space is as in Theorem \ref{suff_condn}. Then we cannot exclude any of the types I, N or E. We know that so-called autocatalytic CRNs are at least of type E or of type I \cite{hoesslysta}. Besides this, not much is currently known for non-weakly reversible CRNs. We collect some examples and refer to Section \ref{subs_examples} for the computations.

\begin{center}
\begin{tabular}{ |c|c|c|c| }
 \hline
Index  & CRN & deficiency &classification \\

 \hline
 NW1  & $S_1\ce{->} S_2,\quad 2S_2\ce{->} 2S_1 $ &$0$& N  \\
            \hline 
NW2  & $S_1\ce{<=>} S_2$ &$1$& N  \\
&$2S_1\ce{->} 2S_2$&&\\
            \hline
 NW3  & $S_1\ce{->} S_2$ &$1$& N  \\
 &$2S_1\ce{<=>} 2S_2$&&\\
            \hline
 NW4  & $2S_1\ce{->} S_1+S_2$ &$1$& N  \\
 &$2S_1\ce{<=>} 2S_2$&&\\
            \hline 
             NW5  & $2S_1\ce{<-} S_1+S_2$ &$1$& N  \\
 &$2S_1\ce{<=>} 2S_2$&&\\
            \hline 
NW6  & $S_1\ce{<=>} S_2$ &$1$& I  \\
&$2S_1\ce{->} S_1+S_2$&&\\
            \hline

NW7  & $S_1\ce{<=>} S_2$ &$1$& I  \\
&$2S_1\ce{<-} S_1+S_2$&&\\
            \hline
  NW8  & $S_1\ce{<=>} S_2$ &$2$& I  \\
&$2S_1\ce{<-} S_1+S_2\ce{->}2S_2$&&\\
  \hline

NW9  & $S_1\ce{<=>} S_2$ &$2$& I  \\
&$2S_1\ce{<-} S_1+S_2\ce{<-}2S_2$&&\\
            \hline
NW10  & $S_1\ce{<=>} S_2$ &$2$& I  \\
&$2S_1\ce{->} S_1+S_2\ce{<-}2S_2$&&\\
            \hline
NW11  & $
 S_1 \ce{<=>}S_2 \ce{<=>} S_3$
  &$4$& E  \\
  &$2S_1 \ce{<-} S_1+S_2\ce{->}2S_2$&&\\
   &$2S_2 \ce{<-} S_2+S_3\ce{->}2S_3$&&\\
            \hline
   
NW12  & $
 S_1 \ce{->}S_2 \ce{->} S_3 \ce{->}S_1$
  &2& E  \\
  &$S_1+S_2 \ce{->} 2S_2,S_2+S_3\ce{->}2S_3$&&\\
   &$S_1+S_3 \ce{->} 2S_1$&&\\
            \hline
\end{tabular}
\end{center}
 Note that NW4, NW5 have slightly different state space decomposition, i.e. the conditions from Theorem \ref{suff_condn} are only necessary for product form.
 
The so-called small number effect appears in NW7 and NW6, which correspond to \cite[Motif F]{Saito2} and \cite[Motif G]{Saito2}. NW1 was already briefly studied concerning its stationary distribution on some irreducible components in \cite[Example 2]{Cappelletti}, and it follows from our study that this is the smallest CRN with non-product-form stationary distribution independently of the rate.
NW9 is an example of a simple autocatalytic CRN on 2 species \cite{hoesslysta}, while NW11 is an instance of an autocatalytic CRN on 3 species, and an inclusion process. NW12 was studied via simulations concerning stochastic oscillations in \cite{Spieler_stoch_osc}, besides being a particular instance of the inclusion process where dynamics can only move in one direction (see Section \ref{icl} or Section \ref{more_gen_PS}).

The non-existence proofs of product-form stationary distribution for NW1-NW5 can be found in Section \ref{subs_examples}, where also derivations for E of NW11, NW12 are given.
Furthermore, based on the pattern of the reactions we conjecture the following to hold.
\begin{conjecture}
Consider a CRN $\cG$ that is conservative, almost essential and non-weakly reversible under Mass-action kinetics. Then there is a connected component $\cR\p\subseteq \cR$ that is not weakly reversible. Assume there is $r\in\cR\p$ such that there is no $\tilde{r}\in\cR\setminus \cR\p, \tilde{r}\neq r$ in a weakly reversible component such that $\nu _r\p-\nu _r=\nu _{\tilde{r}}\p-\nu _{\tilde{r}}$. Then $\cG$ has non-product-form stationary distribution independently of the rate, i.e. it is in N.
\end{conjecture}

\subsection{Other CTMC models}\label{subs_other_m}
 We consider next some classes of examples of interest, first treating CTMCs in particle systems from statistical mechanics and then Queuing networks. Note that we do not aim to be comprehensive in scope, instead we sketch cases of interest where states $x\in\Z^n_{\geq 0}$ transition to $x-e_i+e_j$.  Recall from Section \ref{h}, Section \ref{semi}, that the requirements for our treatment are linear parametrisations in the generator and the state space decomposition.
\subsubsection{Inclusion process}\label{icl}
The inclusion process \cite{Giardina2009,hoesslysta} is a particle system where particles can move from one site of a lattice  to another.
The CTMC dynamics evolves in 
$\Z_{\geq 0}^{\cS}$, where $\cS$ is a finite number of sites (or set of species) which, as a CTMC, is defined by a generator $\mathcal{L}$ of the form
$$\mathcal{L}h(x)=\sum_{i\ne j} p_{ij}x_i (\frac{m}{2}+x_j)(h(x+e_j-e_i)-h(x)).$$

 As observed in \cite{hoesslysta} the inclusion process can be expressed as a stochastic CRN, which is represented with sets of reactions $R_{ij}$ given by
$$
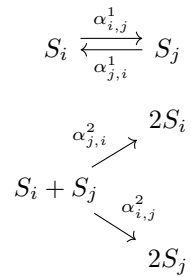
\begin{tikzcd} [ row sep=1em,
  column sep=1em]
S_i\arrow[r, shift left=1ex,"\a_{i,j}^1"]  &S_j \arrow[l,"\a_{j,i}^1"]\\ 
  &2S_i \\ 
S_i+S_j \arrow[ru, "\a_{j,i}^2"] \arrow[dr,"\a_{i,j}^2"] \\
& 2S_j
\end{tikzcd}
$$
The homogeneous case is a special case of the Misanthrope process on a finite lattice \cite{Cocozza-Thivent1985}, and for $p_{ij}=p_{ji}$ (where $\a_{i,j}^1 = p_{ij}\frac{m}{2}$ and $\a_{i,j}^2=p_{ij}$) it corresponds to the symmetric inclusion process, which is generalised by autocatalytic CRNs \cite{hoesslysta}.

 These systems admit product-form stationary distributions for different parts of parameter space (see, i.e., \cite{inclusion_proc} or \cite{hoesslysta}), however, sufficient and necessary conditions on the rates are not known for product-form. After reparametrising inclusion processes, by Theorem \ref{semi_alg} the set of rate space where the processes have product-form stationary distribution is semi-algebraic.

While this is nice, in practice it is nonetheless challenging to use the approach we derived, see NW8, NW11, NW12 for examples.
\subsubsection{More general particle systems}\label{more_gen_PS}
Consider CTMC dynamics on a finite digraph of sites (i.e. $G=(\Lambda,E)$ strongly connected with $|\Lambda|<\infty$), where particles can jump from one site to another if the two vertices are connected in that direction under closed but inhomogeneous dynamics. Consider the dynamics given by the generator 
$$\mathcal{L}h(x)=\sum_{i\to j} b_{ij}(x_i,x_j)(h(x+e_j-e_i)-h(x)),$$
where $b_{ij}$ are functions from $\Z_{\geq0}\times \Z_{\geq0}\to\R_{\geq 0}$ which satisfy \begin{equation}\label{eq_irr}
b_{ij}(n,m)=0\iff n=0,
\end{equation} which we consider with linear parametrisation. This allows general types of functions $b_{ij}$, allowing for inhomogeneities which can be spatial or coming from different types of transition of one site to another (see, e.g. \cite{hoesslysta} or \cite{Grossk_overv}). 
The setting above incorporates different well-known examples and generalisations, among which are the following:
\begin{itemize}
\item Zero-range process: $b_{ij}(n,m)=r_{ij}f(n)$ where $r_{ij}$ is a parameter.
\item Target process: $b_{ij}(n,m)=\1_{n>0}r_{ij}g(n)$ where $r_{ij}$ is a parameter.
\item Inclusion process: $b_{ij}(n,m)=n(d+m)$, where $d$ is a parameter. Note that e.g. symmetric autocatalytic CRNs \cite{hoesslysta} with molecularity two representing the inclusion process
have functions $b_{ij}(n,m)=\a_{ij}^1n+\a_{ij}^2nm$.
\end{itemize}
Corresponding CTMCs are popular models for particle systems, and often admit product-form stationary distributions. We refer to \cite{Grosskinsk_inhom,inclusion_proc,Grossk_overv} for more on the setting we consider and to \cite{kipnis1999sli,Evans1,Schadschn_cx,liggett} for more on context. 
In examples where the $b_{ij}$ satisfy equation \eqref{eq_irr}, after choosing the graph structure and the constants (i.e. the functions $f,g,v$) with linear parametrisation, the subset of parameter space where the CTMC has product-form stationary distribution is semi-algebraic (by Theorem \ref{semi_alg}).

\subsubsection{Whittle networks}
In queuing networks, the particles represent customers which are in a location (or other entities). A fairly broad class of such CTMCs are Whittle networks \cite{serfozo}, where we again focus on the case of closed dynamics on a finite digraph of nodes (i.e. $G=(M,E)$ a strongly connected digraph with $|M|<\infty$). The generator is given by
$$\mathcal{L}h(x)=\sum_{i\to j} \la_{ij}\phi_{i}(x)(h(x+e_j-e_i)-h(x)),$$
where $ \la_{ij}$ are linear coefficients, and again reasonable conditions are usually imposed on the functions $\phi_{i}$ (i.e. $\phi_{i}(x)=0\iff x_i=0$). Again in such a setting, it follows that, after choosing the graph and the constants (i.e. the functions $\phi_{i}$) with linear parametrisation, the subset of rate space with product-form stationary distributions is semi-algebraic (by Theorem \ref{semi_alg}). Furthermore, more general dynamics could be considered as well, i.e., e.g., the functions could be $\phi_{ij}$ instead of $\phi_{i}$. We refer to \cite[Theorem 1.15]{serfozo} or \cite[IV Section 2, Section 5]{asmussen2003applied} for different known sufficient conditions for product-form stationary distributions.

\section{Discussion}\label{disc}
We developed, analysed and applied algebraic approaches for product-form stationary distributions to CRNs and other models from applied probability. While the algebraic characterisation is only valid for CTMC dynamics with special state space decompositions and linear parameters, it applies to many examples and gives a procedure to find the region of the reaction rate (parameter) space where there is product-form. The approach can potentially be extended to more complicated state space decompositions, as long as all irreducible components are finite, or to more general $\I$ product-form stationary distributions. Besides the applicability in CRN theory, we outlined other models of applied probability where the set of rates with product form stationary distribution is semi-algebraic.

We analysed some examples of CRNs under stochastic Mass-action kinetics, and found the smallest CRN with no product form stationary distribution. Weakly reversible CRNs always have some reaction rates that are complex balance, leading to product-form stationary distribution. Therefore the set of rates giving product form is always non-empty for them. Among the examples considered here, this set is either equal to the toric variety described in \cite{sturmfels}, or equals the whole  positive orthant (i.e. of type I and with  $\delta\geq 1$). On the other hand, our investigation into non-weakly reversible CRNs showed that corresponding examples can be of type I, N or E, where further correspondence to particle systems was observed for different examples.

However, finding the equations defining the semi-algebraic set with the proposed method appears to be cumbersome for big examples, and in practice we do not know at which point the ideal chains in Definition \ref{product_ideal} stabilise in general. In terms of expansions of the approach it would be interesting to be able to extend statements from small CRNs to bigger CRNs. In the case of product-form and some specific non-product-form stationary distributions sufficient conditions for inheritance were given in \cite{hoesslyuni}. It is natural to wonder whether sufficient conditions can be given for inheritance of non-product-form stationary distribution without using the specific form of the stationary distribution. In addition to the conjectured statements it could be of interest to extend the approach developed here to more general CRNs (i.e. more general state space decompositions), or to systematically analyse examples for CRNs (for different types of kinetics), models of statistical mechanics or stochastic networks of interest (i.e. for different configurations). In terms of CRNs it would be interesting to know which weakly reversible CRNs are complex balanced if and only if they have product form stationary distribution besides the ones with deficiency zero. More generally, it would be interesting to know structural conditions which imply either non-product-form stationary distributions independently of the rate, or product-form stationary distribution independently of the rate.

\section*{Acknowledgement}
The authors thank Jan Draisma, Christian Mazza, Manoj Gopalkrishnan and Elisenda Feliu for helpful discussions, as well as the reviewers for useful suggestions. LH is supported by the Swiss National Science Foundations Early Postdoc.Mobility grant (P2FRP2\_188023). BPE acknowledges funding from the European Union's Horizon 2020 research and innovation programme under the Marie Sklodowska-Curie IF grant agreement No 794627.

\section{Appendix}\label{app}
\subsection{Some properties of product-form distribution}\label{app_prod_f}
\hfill\break
In this part we take a slightly more abstract view. Consider an indexed set of pairwise disjoint subsets $\Ga_l\subseteq \Z_{\geq 0}^n$ with probability distributions on these sets $\pi_l\in \Delta_{\Ga_l}$, and $l\in\I$
which overall we denote by $(\Ga_l,\pi_l)_{l\in \mathcal{I}}$. We say this set has product-form distribution if there are product-form functions $f_i: \N\to \R_{>0}; S_i\in \cS$ such that for any subset $\Ga_l$ the distribution has the form
\begin{equation}\label{eq:prod_form_app}\pi_{\Ga_l}(x)=Z_{\Ga_l}^{-1}\prod_{S_i\in\cS}^{ } f_i(x_i)\end{equation}
where $$Z_{\Ga_l}=\sum_{x\in \Ga_l}\prod_{S_i\in\cS}^{ } f_i(x_i)$$ is the finite normalising constant.
The derivations to come are consequences of the  following observation.
\begin{remark}\label{basic_rem}
Assume $(\Ga_l,\pi_l)_{l\in \mathcal{I}}$ has product form distribution. Then, if both $x\in \Ga_l$ and $x+e_k\in \Ga_{l\p}$, it follows from \eqref{eq:prod_form_app} that
\begin{equation}\label{basic_prod} \frac{\pi_{\Ga_l}(x)}{\pi_{\Ga_{l\p}}(x+e_k)}=\frac{Z_{\Ga_{l\p}}}{Z_{\Ga_l}}\frac{f_k(x_k)}{f_k(x_k+1)}\end{equation}
\end{remark}
Our next result is the key to Proposition \ref{suff_cond}, and follows from the definition of product form. 
\begin{lemma} \label{Lemma_nec}
Assume $(\Ga_l,\pi_l)_{l\in \mathcal{I}}$ with $\pi_l\in \mathring\Delta_{\Ga_l}$ has product form distribution. Then the following two conditions are satisfied:
\begin{enumerate}
\item[(a)]\label{a} Assume that $x,x+e_j-e_k\in \Ga_i$ and $y,y+e_j-e_k\in \Ga_{i+1}$ are such that $x_j=y_j,x_k=y_k$. Then
$$\frac{\pi_{i}(x+e_j-e_k)}{\pi_i(x)}=\frac{\pi_{i+1}(y+e_j-e_k)}{\pi_{i+1}(y)}$$
\item[(b)] Assume that $x\in \Ga_i;\quad x+e_j,y\in \Ga_{i+1}$ and $y+e_j\in \Ga_{i+2}$ are such that $x_j=y_j$. Then
$$\frac{\pi_{i+1}(x+e_j)}{\pi_i(x)}=a_i\frac{\pi_{i+2}(y+e_j)}{\pi_{i+1}(y)}$$
where $a_i$  is a constant that only depends on the index $i\in \mathcal{I}$.
\end{enumerate}
In particular, these conditions are necessary for product form distribution.
\end{lemma}

It is possible to express the previous conditions which were defined via fractions as products. Together with the fact that the constants $a_i$ only depend on the irreducible component $\Gamma $, this gives the following observation:
\begin{remark}\label{eq_lem}
Equivalent conditions to the ones of Lemma \ref{Lemma_nec} that must hold for product form stationary distributions are the following:
\begin{enumerate}
\item[(a)] Assume that $x,x+e_j-e_k\in \Ga_i$ and $y,y+e_j-e_k\in \Ga_{i+1}$ are such that $x_j=y_j,x_k=y_k$. Then
$$\pi_{i}(x+e_j-e_k)\pi_{i+1}(y)=\pi_{i+1}(y+e_j-e_k)\pi_i(x)   $$
\item[(b)] Assume that $x,z\in \Ga_i;\quad x+e_j,y,z+e_k,w\in \Ga_{i+1}$ and $y+e_j,w+e_k\in \Ga_{i+2}$ are such that $x_j=y_j$ and $z_k=w_k$. Then
$$\pi_{i+1}(x+e_j)\pi_{i+1}(y)\pi_{i+2}(w+e_k)\pi_i(z)=\pi_{i+1}(z+e_k)\pi_{i+1}(w)\pi_{i+2}(y+e_j)\pi_i(x)$$
\end{enumerate}

\begin{remark}\label{eq_more_gen}
More general conditions can be given by replacing $i+1,i+2$ by arbitrary indices $i\p,i\pp$ in the formulation of Lemma \ref{Lemma_nec} and Remark \ref{eq_lem}, but then they become less transparent. Furthermore we will only use them here in the form of Lemma \ref{Lemma_nec}. Note that, depending on the subsets $\Ga_l$ of $\Z_{\geq 0}^n$, the conditions can be empty.
\end{remark}

\end{remark}

\begin{proposition}\label{suff_cond}  Consider $(\Ga_l,\pi_l)_{l\in \mathcal{I}}$, where $$\Ga_l=\{x\in \Z^\cS_{\geq 0}|\sum_{S_i\in \cS} x_i=l\}$$  and $\pi_l\in \mathring\Delta_{\Ga_l}$. Assume that $\mathcal{I}=\Z_{\geq 1}$. Then the two conditions of Lemma \ref{Lemma_nec} are sufficient for product-form.
\end{proposition}

\begin{proof}

We assume $n\geq 2$, and instead of Lemma \ref{Lemma_nec}, we use the equivalent conditions in Remark \ref{eq_lem}. Let $l=1$ and let us denote $\alpha _i=\pi _1(e_i)$ for $i=1,\ldots ,n$, defining completely the stationary distribution on $\Gamma _1$. Then $\sum _{i=1}^n \alpha _i=1$. For the stationary distribution to be of product form we require, at this level, the existence of functions $f_i:\mathbb{N}\longrightarrow \mathbb{R}_{>0}$ for $i=1,\ldots ,n$ such that 
\begin{equation}\label{eq:prop_l=1}\pi _1(e_i)=\frac{\left( \prod _{j\neq i}f_j(0)\right) f_i(1)}{Z_{\Gamma _1}}\mbox{,}\end{equation}
where $Z_{\Gamma _1}=\sum _{x\in \Gamma _1}\prod _{i=1}^nf_i(x_i)$. We can choose $f_i$ for $i=1,\ldots ,n$ mapping $0$ to $1$ and $1$ to $\alpha _i$, obtaining $Z_{\Gamma _1}=\sum _{i=1}^n\alpha _i=1$ such that the product-form functions satisfy equation (\ref{eq:prop_l=1}) for $i=1,\ldots ,n$.

For $l>1$ we will show that, if we assume the algebraic conditions imposed on the stationary distribution for being product form (i.e. as in Lemma \ref{Lemma_nec}), we can recursively define 
\begin{equation}\label{eq:def_f_i}f_i(l)=Z_{\Gamma _l}\pi _l(le_i)\end{equation}
where
\begin{equation}\label{eq:def_Z_l}Z_{\Gamma _l}=Z_{\Gamma _{l-1}}\frac{\pi _{l-1}(x)}{\pi _l(x+e_k)}\frac{f_k(x_k+1)}{f_k(x_k)}\end{equation}
for any $k\in [n]$ and $x\in \Gamma _{l-1}$ such that $x\neq (l-1)e_k$, in such a way that: (1) the definition of $Z_{\Gamma _l}$ does not depend on the choice of $x$ (as long as $x\neq (l-1)e_k$), and (2) the $f_i$ give product form stationary distribution for $\Ga_l$ (i.e. as in equation \eqref{eq:prod_form_app}).

We proceed by 2-step induction on $l$: Let $l=2$. Using the special structure assumed for the irreducible components, the stationary distribution is determined by $\pi _2(2e_i)$ for $i=1,\dots ,n$, and $\pi _2(e_i+e_k)$ for $i\neq k$. Requiring that it has product form imposes the following conditions: 
$$\pi _2(2e_i)=\frac{\left( \prod _{j\neq i}f_j(0)\right) f_i(2)}{Z_{\Gamma _2}}\mbox{,\; }\pi _2(e_i+e_k)=\frac{\left( \prod _{j\neq i,k}f_j(0)\right) f_i(1)f_k(1)}{Z_{\Gamma _2}}\mbox{.}$$
Note that for any $k$ and for any $x=e_i$ where $i\neq k$ in equation \eqref{eq:def_Z_l}, we obtain $Z_{\Gamma _2}=\frac{\alpha _i\alpha _k}{\pi _2(e_i+e_k)}$ according to the proposed definition. Let $z=e_j$ for any $j\neq k,i$. Then $x,z\in \Gamma _1$, $x_k,z_k<1$ (so $x$ and $z$ can both be used for \eqref{eq:def_Z_l}), and Remark \ref{eq_lem} (a), when applied to $x\in \Gamma_1$ and $y=x+e_k\in \Gamma _2$ (where $x_s=y_s$ for any $s\neq k$), ensures $\pi _1(e_j)\pi _2(e_i+e_k)=\pi _2(e_j+e_j)\pi _1(e_i)$, which implies
$$\frac{\alpha _i\alpha _k}{\pi _2(e_i+e_k)}=\frac{\alpha _j\alpha _k}{\pi _2(e_j+e_k)}\mbox{,}$$
proving that $Z_{\Gamma _2}$ is well defined. Therefore we only need to define $f_i(l)=Z_{\Gamma _2}\pi _2(2e_i)$ as in \eqref{eq:def_f_i} for each $i\in [n]$, which gives product-form stationary distribution.

We assume now that we have defined $f_i(l)$ for all $l\leq m+1$ as in \eqref{eq:def_f_i} by means of $Z_{\Gamma _l}$ and $\pi _l(le_i)$, while $Z_{\Gamma _l}$ itself is defined as in \eqref{eq:def_Z_l} in terms of $\Z_{\Gamma _{l-1}}$, $\pi _l$, $\pi _{l-1}$ and the $f_k(0),\ldots ,f_k(l-1)$ for $k\in [n]$, and that \eqref{eq:prod_form_app} holds for these. Let us show that then $Z_{\Gamma _{m+2}}$ can be defined as in \eqref{eq:def_Z_l} and does not depend on the choice of $k$ and $x$, and that choosing then $f_i(m+2)=Z_{\Gamma _{m+2}}\pi _{m+2}((m+2)e_i)$ for $i\in [n]$ gives the desired functions $f_1,\ldots ,f_n$ for the product form of $\pi _{m+2}$ as in \eqref{eq:prod_form_app}.

Choose any $w\in \Gamma _{m+1}$, and let $k$ be such that $w_k<m$, so that $w+e_k\in \Gamma _{m+2}$ can be chosen for \eqref{eq:def_Z_l}. Let $y\in \Gamma _{m+1}$ be different than $w$ and let $j$ be such that $y_j<m$, so $y+e_j\in \Gamma _{m+2}$ can be chosen for \eqref{eq:def_Z_l} too. Then there exist some $x,z\in \Gamma _m$ such that $x_j=y_j$ and $z_k=w_k$. We claim that
\begin{equation}\label{eq:claim}Z_{\Gamma _{m+1}}\frac{\pi _{m+1}(y)}{\pi _{m+2}(y+e_j)}\frac{f_j(y_j+1)}{f_j(y_j)}=Z_{\Gamma _{m+1}}\frac{\pi _{m+1}(w)}{\pi _{m+2}(w+e_k)}\frac{f_k(w_k+1)}{f_k(w_k)}\mbox{.}\end{equation} 

The induction hypothesis ensures that $$Z_{\Gamma _{m+1}}=Z_{\Gamma _m}\frac{\pi _m(x)}{\pi _{m+1}(x+e_j)}\frac{f_j(x_j+1)}{f_j(x_j)}=Z_{\Gamma _{m}}\frac{\pi _m(z)}{\pi _{m+1}(z+e_k)}\frac{f_k(z_k+1)}{f_k(z_k)}\mbox{,}$$
because $x$ and $z$ can also be chosen in \eqref{eq:def_Z_l}, so
$$\frac{\pi _m(x)\pi _{m+1}(z+e_k)}{\pi _{m+1}(x+e_j)\pi _{m}(z)}=\frac{f_k(z_k+1)f_j(x_j)}{f_j(x_j+1)f_k(z_k)}\mbox{.}$$
But Remark \ref{eq_lem} (b) makes the left hand side equal to $\frac{\pi _{m+1}(y)\pi _{m+2}(w+e_k)}{\pi _{m+2}(y+e_j)\pi _{m+1}(w)}$ and, by the way in which $x$ and $z$ have been chosen, the right hand side equals $\frac{f_k(w_k+1)f_j(y_j)}{f_j(y_j+1)f_k(w_k)}$, proving \eqref{eq:claim}.

To finish, note that the corresponding product-form functions solve the ME on $\Ga_{m+2}$, because the constructed functions $f_i$  give the values of $\pi (x) $ also for $x\neq (m+2)e_1$: this follows from the relation in Remark \ref{eq_lem} b), which is guaranteed to hold between the $\pi (x)$ for different points $x$ of $\Gamma _{m+2}$.  

\end{proof}

\begin{remark}\label{rem_2} If $|\cS|=2$, then the two conditions of Lemma \ref{Lemma_nec} are sufficient for product-form even if $\mathcal{I}=\Z_{\geq 2}$. To see this, it is enough to show that the Ansatz for product-form stationary distribution has a positive real solution if considered on $\Ga_2,\Ga_3$, e.g. by using $log(\cdot)$ and the Rouch\'e-Capelli Theorem, from which we deduce that there are infinitely many solutions. Note that this does not hold for arbitrary index sets $\mathcal{I}=\Z_{\geq j}$. Furthermore for $|\cS|=2$ only relations of Remark \ref{eq_lem} (b) have to be considered. 
\end{remark}

\subsection{Derivations of conclusions for the examples of Section \ref{sect_ex}}\label{subs_examples}
For the convenience of the reader, we give some details on the analysis of the examples considered in Section \ref{sect_ex}. We start with the analysis of the semi-algebraic sets, and then cover the types of the product-form functions for the product-form stationary distributions of the examples of type I.

\subsubsection{Analysis of the semi-algebraic sets for CRNs}
Let us partially present here the computations involving the examples in Section \ref{sect_ex}. For this, we consider the network in \eqref{eq:big_RN}, such that all the examples on two species from Section \ref{sect_ex} are subnetworks. The Master equation with the relevant reactions is
\begin{align}\label{eq:ME_gen_ex}
\1_{x\geq (1,0)} & \left[ \pi (x_1,x_2)\alpha x_1-\pi (x_1-1,x_2+1)\beta (x_2+1)\right] + \\
\1_{x\geq (0,1)} & \left[ \pi (x_1,x_2)\beta x_2 -\pi (x_1+1,x_2-1)\alpha (x_1+1)\right] + \nonumber \\
\1_{x\geq (1,1)} & \left[ \pi (x_1,x_2)\lambda _5x_1x_2-\pi (x_1+1,x_2-1)\lambda _3(x_1+1)x_1\right] +\nonumber \\
\1_{x\geq (2,0)} & \left[ \pi (x_1,x_2)(\lambda _1+\lambda _3)x_1(x_1-1)-\pi (x_1-2,x_2+2)\lambda _2(x_2+2)(x_2+1) \right. \nonumber \\
& \left. -\pi (x_1-1,x_2+1)\lambda _5(x_1-1)(x_2+1)\right] + \nonumber \\
\1_{x\geq (0,2)} & \left[ \pi (x_1,x_2)\lambda _2x_2(x_2-1)-\pi (x_1+2,x_2-2)\lambda _1(x_1+2)(x_1+1)\right] =0 \nonumber
\end{align}

We will next focus on the computations for W3, W4, NW2, NW3, NW4, NW5, for which it will be sufficient to consider the corresponding ideals. For each of them the corresponding Master equation can be obtained by setting in \eqref{eq:ME_gen_ex} the rates of the  reactions that are not present to zero. Then we determine their irreducible components, compute the kernel of the matrices $A(\kappa )$ using the algorithm given in Section \ref{h}, and obtain expressions for the ideals $I_l$ defined in Section \ref{app_prod_fn}. For the computations we used Maple, and even though we did not perform an analysis on the computational cost, for the examples considered here this was not an important issue: they took less than 3 seconds on the biggest example (including kernel computation, definition of $J_i$ ideals and computation of $I_i$ for $i\leq 5$). 

We next consider the irreducible components which, for these examples, have the following index sets (i.e. the $\Gamma _l, l\in \mathcal{I}$ are as in Theorem \ref{suff_condn}):
\begin{itemize}
\item $\mathcal{I}=\mathbb{Z}_{\geq 1}$:W3, NW2, NW3.
\item $\mathcal{I}=\mathbb{Z}_{\geq 2}$:W4.
\item $\mathcal{I}=\mathbb{Z}_{\geq 3}$:NW4, NW5.
\end{itemize}
We note here that, while Theorem \ref{suff_condn} does not apply to NW4, NW5 as the index sets are different, the ideals give necessary conditions for product form stationary distribution by Lemma \ref{Lemma_nec}. Furthermore for W3, NW2, NW3 and W4 we know by Theorem \ref{suff_condn} and Remark \ref{rem_2} that the conditions in $I_N$ for $N\in \mathbb{N}$ such that $I=I_N$, are also sufficient conditions.

As a consequence of the irreducible decompositions, $I_3$ is trivial for W4, NW4 and NW5, and the ideal $I_4$ is still trivial for NW4 and NW5, forcing us to compute further ideals in order to obtain nontrivial necessary conditions for product form, according to Lemma \ref{Lemma_nec}.

We next indicate for each of the mentioned examples the features that allows us to classify them as E/N form Section, as in  Section \ref{sect_ex}, based on these computations. The ideals $J_j$ can be written in terms of the $\pi (x)$, and this expression depends only on the irreducible component. Since the structure of the irreducible components is shared for all networks we are considering (for most of them the index set starts in 2 or 3), there are common expressions:
\begin{equation*}
J_3= \langle \pi _1(0,1)\pi _2(2,0)\pi _3(1,2)-\pi _1(1,0)\pi _2(0,2)\pi _3(2,1) \rangle 
\end{equation*}
\begin{align*}
J_4= J_3 + & \langle \pi _2(1,1)\pi _3(3,0)\pi _4(2,2)-\pi _2(2,0)\pi _3(1,2)\pi _4(3,1), \\
& \pi _2(0,2)\pi _3(2,1)\pi _4(1,3)-\pi _2(1,1)\pi _3(0,3)\pi 4(2,2),\\
& \pi _2(0,2)\pi _3(2,1)\pi _3(3,0)\pi _4(1,3)-\pi _2(2,0)\pi _3(1,2)\pi _3(0,3)\pi _4(3,1)\rangle 
\end{align*}
\begin{align*}
J_5 = J_4 + & \langle \pi _3(2,1)\pi _4(4,0)\pi _5(3,2)-\pi _3(3,0)\pi _4(2,2)\pi _5(4,1),\\
& \pi _3(1,2)\pi _4(3,1)\pi _4(4,0)\pi _5(2,3)-\pi _3(3,0)\pi _4(2,2)\pi _4(1,3)\pi _5(4,1),\\
& \pi _3(0,3)\pi _4(3,1)\pi _4(4,0)\pi _5(1,4)-\pi _3(3,0)\pi _4(1,3)\pi _4(0,4)\pi _5(4,1),\\
& \pi _3(1,2)\pi _4(3,1)\pi _5(2,3)-\pi _3(2,1)\pi _4(1,3)\pi _5(3,2),\\
& \pi _3(0,3)\pi _4(2,2)\pi _4(3,1)\pi _5(1,4)-\pi _3(2,1)\pi _4(1,3)\pi _4(0,4)\pi _5(3,2),\\
& \pi _3(0,3)\pi _4(2,2)\pi _5(1,4)-\pi _3(1,2)\pi _4(0,4)\pi _5(2,3) \rangle
\end{align*}

After substituting the $\pi _l(x)$ by elements in the kernel of $A(\kappa )$ for each example (i.e. by the $F_1(\k), \ldots ,F_m(\k)$ of Proposition \ref{clm2}) we obtain the corresponding ideals $I_3$, $I_4$, $I_5$, and the following observations can be made:
\begin{itemize}
\item In example NW2, the ideal $I_3$ contains the polynomial $\beta ^5\lambda _1^2(\alpha +\lambda _1)$, which has no solutions in the positive orthant of the rate space $\mathbb{R}^4_{>0}$. Hence NW2 is of type N.
\item For example NW3, $I_3$ contains a polynomial of the form $\lambda _2^3(\alpha + 2\lambda _1)$, making NW3 of type N.
\item For NW5, as we mentioned already, $I_3=I_4=(0)$, so we need $I_5$ for the first necessary conditions. It turns out again that $I_5$ contains a polynomial with only positive coefficients, which has also no positive solutions in the rate space $\mathbb{R}^3$.
\item For W3, we obtain the condition $\alpha ^2\lambda _2-\beta ^2\lambda _1=0$ already in $I_3$. Furthermore we note that the above condition is the condition for complex balance, so $M_\cG\subset I_3$  (product form implies complex balanced) and $\emptyset\neq V_{>0}(M_\cG)= V_{\cG,>0}$ by Lemma \ref{comp_toric}. 
\item For W4, $I_3$ is empty, but $I_4$ contains the element $\lambda _1\lambda _5^2-\lambda _2\lambda _3^2$, which must vanish at any set of rates potentially having product form stationary distribution. We note that the previous condition ensures complex balance for the CRN. So again by Lemma \ref{comp_toric}, $\emptyset\neq V_{>0}(M_\cG)= V_{\cG,>0}$. 
\item Finally, NW4 has trivial ideals $I_3$ and $I_4$, but $I_5$ yields some necessary conditions in this case, which include $6\lambda _1 - 6\lambda _2+7\lambda _3=0$, $\lambda _3^2+(2\la_1-6\la_2)\la_3+\la_1(\la_1-5\la_2)=0$.
In the following we prove that this implies that NW4 is of type N.

W.l.o.g., let $\la_3=1$ (the polynomials are homogeneous). Then the first condition gives that $\la_2=\frac{6\la_1+7}{6}$, the second gives $1+(2\la_1-6\la_2)+\la_1(\la_1-5\la_2)=0$. 
Hence we can insert the expression for $\la_2$ in the other, and we get the polynomial
$24\la_1^2+59\la_1+36$. This has no positive solution.
\end{itemize}

To complete our analysis, we note that NW11, NW12 are known as particle systems to have values with product-form stationary distribution (i.e. they are I or E), see \cite[Theorem 2.1]{inclusion_proc} or \cite{hoesslysta}. They both have irreducible components $\Gamma _l, l\in \Z_{\geq 1}$ as in Theorem \ref{suff_condn}. After a check of $I_2$ it is easy to find nontrivial elements in the ideal showing that not all values give product-form stationary distributions (i.e. it is enough to look at, e.g. $\pi_{1}(0,0,1)\pi_{2}(1,1,0)-\pi_{1}(1,0,0)\pi_2(0,1,1)$ to see that for some values of $\k$ it does not vanish).

\subsubsection{CRNs with product-form stationary distributions independently of the rate}\hfill\break
We next consider the examples of type I. As solutions for stationary distributions for such examples were already covered in depth in  \cite{hoesslyuni}, our presentation will be relatively brief.
\begin{itemize}
\item
Weakly reversible examples:\\
W1 is clear, W2 is from \cite[Example 5.1]{hoesslyuni}, and similarly W6 follows as it is a union of CRNs with compatible product-forms (see \cite[Remark 5.3]{hoesslyuni}).\item
Non-weakly reversible examples:\\
NW6-NW10 follow as in \cite[Example 5.1]{hoesslyuni}.
\end{itemize}
We next give the product-form functions for the stationary distributions,
$$g_1(m)=\frac{a_0^m}{m!},\quad g_2(m)=\frac{1}{m!}\prod_{l=1}^{m} \frac{a_1+a_2(l-1)}{a_3+a_4(l-1)},$$
$$g_3(m)=\frac{1}{m!}\prod_{l=1}^{m} \frac{a_5}{a_6+a_7(l-1)},\quad g_4(m)=\frac{1}{m!}\prod_{l=1}^{m} \frac{a_8+a_9(l-1)}{a_{10}},$$
where the $a_i$ are positive constants.
Then we classify the CRNs with type I according to the product-form functions in their product-form stationary distributions in the following table. The table contains the product-form functions for each species, where e.g. the first row signifies that CRN $W1$ has product-form stationary distribution 
$$ \pi(x)=\frac{1}{Z}f_1(x_1)f_2(x_2),$$
where both $f_1,f_2$ are of the form $g_1$ as above. We note that only the form is the same, but parameters might differ.
\begin{center}
\begin{tabular}{ |c|c|c|c|c| }
 \hline
Index  & in $S_1$ & in $S_2$ & in $S_3$& in $S_4$ \\
 \hline
W1  & $g_1$ &$g_1$& $-$& $-$  \\
            \hline
W2  & $g_2$ &$g_1$&$-$& $-$  \\
            \hline
W5  & $g_2$ &$g_1$&$g_2$& $-$  \\
            \hline    
W6  & $g_2$ &$g_1$&$g_2$& $g_1$  \\
            \hline                      

NW6  & $g_3$ &$g_1$& $-$& $-$   \\
            \hline
NW7  & $g_4$ &$g_1$& $-$& $-$   \\
            \hline
NW8  & $g_4$ &$g_4$& $-$& $-$   \\
            \hline
NW9  & $g_4$ &$g_3$& $-$& $-$   \\
            \hline    
NW10  & $g_3$ &$g_3$& $-$& $-$   \\
            \hline                      
 \end{tabular}
\end{center}

\end{document}